\documentclass[11pt]{amsart}
\usepackage[linkcolor=blue, citecolor = blue]{hyperref}
\hypersetup{colorlinks=true}

\usepackage{amsthm,amsfonts,amsmath,amssymb,amscd} 

\usepackage{verbatim}
\usepackage{float}
\usepackage{wrapfig}
\usepackage{mathtools}
\usepackage[color=orange!20, linecolor=orange]{todonotes}

\usepackage{caption, subcaption}
\usepackage{enumitem}

\usepackage{ytableau}

\usepackage{tikz}
\usetikzlibrary{decorations.markings}
\tikzstyle{vertex}=[circle, draw, inner sep=0pt, minimum size=4pt]

\usetikzlibrary{arrows,automata}
\usepackage{tikz, tikz-3dplot, pgfplots}
\usepackage{tkz-graph}
\usetikzlibrary[positioning,patterns]
\usepackage{tikz-3dplot}
\usepackage{wasysym}

\newtheorem{theorem}{Theorem}[section]
\newtheorem{proposition}[theorem]{Proposition}
\newtheorem{lemma}[theorem]{Lemma}
\newtheorem{corollary}[theorem]{Corollary}
 
\theoremstyle{definition}
\newtheorem{definition}[theorem]{Definition}
\newtheorem{example}[theorem]{Example}

\theoremstyle{remark}
\newtheorem{remark}{Remark}

\newcounter{x}
\newcounter{y}
\newcounter{z}

\newcommand\xaxis{210}
\newcommand\yaxis{-30}
\newcommand\zaxis{90}

\newcommand\topside[3]{
  \fill[fill=white, draw=black,shift={(\xaxis:#1)},shift={(\yaxis:#2)},
  shift={(\zaxis:#3)}] (0,0) -- (30:1) -- (0,1) --(150:1)--(0,0);
}

\newcommand\ttopside[3]{
  \fill[fill=white, draw=black,dashed,shift={(\xaxis:#1)},shift={(\yaxis:#2)},
  shift={(\zaxis:#3)}] (0,0) -- (30:1) -- (0,1) --(150:1)--(0,0);
}

\newcommand\leftside[3]{
  \fill[fill=cyan, draw=black, shift={(\xaxis:#1)},shift={(\yaxis:#2)},
  shift={(\zaxis:#3)}] (0,0) -- (0,-1) -- (210:1) --(150:1)--(0,0);
}

\newcommand\lleftside[3]{
  \fill[fill=white, draw=black,dashed, shift={(\xaxis:#1)},shift={(\yaxis:#2)},
  shift={(\zaxis:#3)}] (0,0) -- (0,-1) -- (210:1) --(150:1)--(0,0);
}

\newcommand\rightside[3]{
  \fill[fill=orange, draw=black,shift={(\xaxis:#1)},shift={(\yaxis:#2)},
  shift={(\zaxis:#3)}] (0,0) -- (30:1) -- (-30:1) --(0,-1)--(0,0);
}

\newcommand\rrightside[3]{
  \fill[fill=white, draw=black,dashed,shift={(\xaxis:#1)},shift={(\yaxis:#2)},
  shift={(\zaxis:#3)}] (0,0) -- (30:1) -- (-30:1) --(0,-1)--(0,0);
}

\newcommand\cube[3]{
  \topside{#1}{#2}{#3} \leftside{#1}{#2}{#3} \rightside{#1}{#2}{#3}
}
\newcommand\ccube[3]{
  \ttopside{#1}{#2}{#3} 
  \lleftside{#1}{#2}{#3} \rrightside{#1}{#2}{#3}
}

\newcommand\cccube[3]{
  \leftside{#1}{#2}{#3} \rightside{#1}{#2}{#3}
}

\newcommand\ocube[3]{
  \fill[fill=white, draw=black,shift={(\xaxis:#1)},shift={(\yaxis:#2)},
  shift={(\zaxis:#3)}] (0,0) -- (30:1) -- (0,1) --(150:1)--(0,0);
}
\newcommand\planepartition[1]{
 \setcounter{x}{-1}
  \foreach \a in {#1} {
    \addtocounter{x}{1}
    \setcounter{y}{-1}
    \foreach \b in \a {
      \addtocounter{y}{1}
      \setcounter{z}{-1}
      \foreach \c in {0,...,\b} {
        \addtocounter{z}{1}
      \ifthenelse{\c=0}{\setcounter{z}{-1},\addtocounter{y}{0}}{
        \cube{\value{x}}{\value{y}}{\value{z}}}
      }
    }
  }
}

\newcommand\pplanepartition[1]{
 \setcounter{x}{-1}
  \foreach \a in {#1} {
    \addtocounter{x}{1}
    \setcounter{y}{-1}
    \foreach \b in \a {
      \addtocounter{y}{1}
      \setcounter{z}{-1}
      \foreach \c in {0,...,\b} {
        \addtocounter{z}{1}
      \ifthenelse{\c=0}{\setcounter{z}{-1},\addtocounter{y}{0}}{
        \ccube{\value{x}}{\value{y}}{\value{z}}}
      }
    }
  }
}

\newcommand\oplanepartition[1]{
 \setcounter{x}{-1}
  \foreach \a in {#1} {
    \addtocounter{x}{1}
    \setcounter{y}{-1}
    \foreach \b in \a {
      \addtocounter{y}{1}
      \setcounter{z}{-1}
      \foreach \c in {0,...,\b} {
        \addtocounter{z}{1}
      \ifthenelse{\c=0}{\setcounter{z}{-1},\addtocounter{y}{0}}{
        \ocube{\value{x}}{\value{y}}{\value{z}}}
      }
    }
  }
}

\title[MacMahon's statistics on higher-dimensional partitions]{
MacMahon's statistics on higher-dimensional partitions
} 

\author[Alimzhan Amanov \and Damir Yeliussizov]{Alimzhan Amanov \and Damir Yeliussizov}

\address{KBTU, Almaty, Kazakhstan}
\email{\href{mailto:alimzhan.amanov@gmail.com}{alimzhan.amanov@gmail.com}, \href{mailto:yeldamir@gmail.com}{yeldamir@gmail.com}}

\begin{document}

\begin{abstract}
We study some combinatorial properties of higher-dimensional partitions which generalize plane partitions. We present a natural bijection between $d$-dimensional partitions and $d$-dimensional arrays of nonnegative integers. This bijection has a number of important applications. We introduce a statistic on $d$-dimensional partitions, called the {corner-hook volume}, whose generating function has the formula of MacMahon's conjecture. 
We obtain multivariable formulas whose specializations give   
analogues of various formulas known for plane partitions.  
We also introduce higher-dimensional analogues of dual Grothendieck polynomials which are quasisymmetric functions and whose specializations enumerate higher-dimensional partitions of a given shape. Finally, we show probabilistic connections with a directed last passage percolation model in $\mathbb{Z}^d$. 
\end{abstract}

\maketitle


\section{Introduction}

{\it Higher-dimensional partitions} 
are classical combinatorial objects introduced by 
MacMahon over a century ago. While the concept itself is a straightforward generalization of the usual {\it integer partitions}, the problems related to it are very challenging. 
For (2-dimensional) {\it plane partitions}, MacMahon obtained his celebrated 
enumerative formulas \cite{macmahon} (cf. \cite[Ch.~7]{sta}). For general $d$-dimensional partitions, he only conjectured a formula of the volume generating function, which was later computed to be incorrect \cite{atkin}. 

Despite long interest and many connections to various fields including algebra, combinatorics, 
geometry, probability and statistical physics, the subject remains rather mysterious---very little is known about $d$-dimensional partitions for $d \ge 3$. 
See \cite{atkin, knuth, gov} on some computational and enumerative aspects;  \cite{mr, bgp, dg} on asymptotic data and connections to physics;  \cite{bbs, nekrasov, ck} on further aspects particularly related to the theory of Donaldson-Thomas invariants. 
(See also the remarks and references in final Sec.~\ref{sec:final}.)

At the same time, the theory of plane partitions has greatly developed, see \cite{andrews, sta, krat} and many references therein.
Its success mainly comes from the theory of symmetric functions, especially by using the Robinson-Schensted-Knuth (RSK) correspondence and Schur polynomials. 
The lack of tools for 
higher-dimensional generalizations makes it difficult to approach them, 
and here one can try to develop  analogous methods. 
This paper is in this direction.

\vspace{0.5em}

Let us summarize our results.


\subsection{Higher-dimensional partitions and matrices} Firstly, we present a natural bijection between $d$-dimensional arrays of nonnegative integers and $d$-dimensional partitions, see Sec.~\ref{sec:bij}. 
Roughly speaking, any $d$-dimensional partition can be viewed as a matrix of largest paths for some 
source weight matrix. The bijection has nice properties which relate natural statistics for both objects. We then give a number of applications. 

\subsection{Corner-hook volume and interpretation of MacMahon's numbers} One of the main consequences of our bijection is the multivariable generating series presented in Theorem~\ref{mult} whose specializations allow to explicitly compute generating functions for certain statistics on $d$-dimensional partitions. In particular, we introduce two statistics on $d$-dimensional partitions: {\it corners} $\mathrm{cor}(\cdot)$ and {\it corner-hook volume} $|\cdot|_{ch}$ (see Sec.~\ref{sec:mac} and \ref{sec:mac1} for definitions) with generating functions shown below. 

\begin{theorem}[Corner-hook generating function, cf. Corollary~\ref{full}]\label{one}
We have the following generating function
$$
\sum_{\pi} t^{\mathrm{cor}(\pi)} q^{|\pi|_{ch}} = \prod_{n = 1}^{\infty} {(1 - t q^n)^{-\binom{n + d - 2}{d-1}}},
$$
where the sum runs over $d$-dimensional partitions $\pi$. 
\end{theorem}
For $d = 2$, this formula is equidistributed with {Stanley's trace generating function} \cite[Thm.~7.20.1]{sta} but the statistics are not identical. MacMahon conjectured \cite{macmahon} that the generating function 
$$
\sum_{n = 0}^{\infty} m_d(n)\, q^n = \prod_{n = 1}^{\infty} {(1 - q^n)^{-\binom{n + d - 2}{d-1}}}
$$
gives the volume generating function $\sum_{\pi} q^{|\pi|}$ for $d$-dimensional partitions. This was shown to be incorrect for $d \ge 3$ 
\cite{atkin}. However, from Theorem~\ref{one} we obtain the following interpretation of {\it MacMahon's numbers} $m_d(n)$, thus `correcting' his guess via the corner-hook volume statistic 
so that 
$$m_d(n) = |\{\text{$d$-dimensional partitions } \pi : |\pi|_{ch} = n\}|.$$

More generally, we also prove results for generating functions over partitions with fixed shape.

\begin{theorem}[Corner-hook generating function with fixed shape, cf. Theorem~\ref{shaped}] Let $\rho$ be a
shape of a fixed $d$-dimensional partition.
We have the following generating function 
$$
\sum_{\mathrm{sh}(\pi) \subseteq \rho} t^{\mathrm{cor}(\pi)} q^{|\pi|_{ch}} = 
\prod_{(i_1,\ldots, i_d) \in \rho} {\left(1 - t q^{i_1 + \ldots + i_d - d + 1}\right)^{-1}},
$$
where the sum runs over $d$-dimensional partitions $\pi$ of shape $\rho$. 
\end{theorem}

\subsection{$d$-dimensional Grothendieck polynomials} To develop tools for studying $d$-dimensional partitions, one might be looking for 
analogues of Schur polynomials whose specializations allow to enumerate them. 
We work in a slightly different direction. In Sec.~\ref{sec:ddg} we define higher-dimensional analogues of {\it dual Grothendieck polynomials}. These new functions are indexed by shapes of $d$-dimensional partitions and in specializations they compute the number of such partitions. 
For $d = 2$, they turn into the dual symmetric Grothendieck polynomials (indexed by partitions) 
known as $K$-theoretic analogues of Schur polynomials introduced in \cite{lp} (see also \cite{dy, dy2} for more on these functions). 

Let us illustrate our results in the special case for (3-dimensional) {\it solid partitions}. 
We define the polynomials (see eq.~\eqref{g3d}) $g_{\pi}(\mathbf{x}; \mathbf{y}; \mathbf{z})$ in three sets of variables indexed by plane partitions $\{\pi\}$. These polynomials enumerate solid partitions within a given shape, e.g. we have
$$
g_{[b] \times [c] \times [d]}(1^{a+1}; 1^{b}; 1^c) = \text{number of solid partitions inside the box $[a] \times [b] \times [c] \times [d]$.}
$$
We show that the following generating series identity holds.
\begin{theorem}[Cauchy-type identity for 3d Grothendieck polynomials, cf. Corollary~\ref{gcauchy}] 
We have 
$$
\sum_{\pi} g_{\pi}(\mathbf{x}; \mathbf{y}; \mathbf{z}) = \prod_{i = 1}^{a} \prod_{j = 1}^{b} \prod_{k = 1}^{c} \frac{1}{1 - x_i y_j z_k},
$$
where the sum runs over plane partitions $\pi$ with shape inside the rectangle $b \times c$. 
\end{theorem}

It is known that dual Grothendieck polynomials (for $d = 2$) are symmetric (in $\mathbf{x}$). As we show, this is no longer the case for $d \ge 3$. However, we prove that these new functions are {\it quasisymmetric}, 
the next 
 known 
class containing symmetric functions (see e.g. \cite[Ch.~7.19]{sta}). 
\begin{theorem}[cf. Theorem~\ref{qsym}]
We have: $g_{\pi}(\mathbf{x}; \mathbf{y}; \mathbf{z})$ is quasisymmetric in $\mathbf{x}$. 
\end{theorem}

\subsection{Last passage percolation in $\mathbb{Z}^{d}$} It turns out that these problems  are closely related to the {\it directed last passage percolation model} with geometric weights in $\mathbb{Z}^d$ (see \cite{martin} for a survey on this probabilistic model). We prove that $d$-dimensional Grothendieck polynomials naturally compute distribution formulas for this model (see Theorem~\ref{lpp}). See Sec.~\ref{sec:lpp} for details.

\section{Preliminary definitions}\label{sec:prelim}

We use the following basic notation: $\mathbb{N}$ is the set of nonnegative integers; $\mathbb{Z}_+$ is the set of positive integers; $\{\mathbf{e}_1, \ldots, \mathbf{e}_d\}$ is the standard basis of $\mathbb{Z}^d$; and $[n] := \{1, \ldots, n \}$.

\vspace{0.5em}

A {\it $d$-dimensional $\mathbb{N}$-matrix} is an array $\left(a_{i_1, \ldots, i_d}\right)_{i_1, \ldots, i_d \ge 1}$ of nonnegative integers 
with only finitely many nonzero elements. 
A {\it $d$-dimensional partition} is a $d$-dimensional $\mathbb{N}$-matrix 
$\left(\pi_{i_1, \ldots, i_d} \right)$ such that 
$$
\pi_{i_1, \ldots, i_d} \ge \pi_{j_1, \ldots, j_d} ~\text{ for }~ i_1 \le j_1, \ldots, i_d \le j_d.
$$
Let $\mathcal{M}^{(d)}$ be the set of $d$-dimensional $\mathbb{N}$-matrices and $\mathcal{P}^{(d)}$ be the set of $d$-dimensional partitions. 
For $\pi = (\pi_{i_1, \ldots, i_d}) \in \mathcal{P}^{(d)}$, the {\it volume} (or size) of $\pi$ denoted by $|\pi|$ is defined as 
$$
|\pi| = \sum_{i_1, \ldots, i_d} \pi_{i_1, \ldots, i_d}.
$$
Any partition $\pi$ is uniquely determined by its {\it diagram} $D(\pi)$ which is the set 
$$D(\pi) := \{(i_1, \ldots, i_d, i) \in \mathbb{Z}^{d+1}_+: 1 \le i \le \pi_{i_1, \ldots, i_d} \}. 
$$
The {\it shape} of $\pi$ denoted by $\mathrm{sh}(\pi)$ is the set 
$$\mathrm{sh}(\pi) := \{(i_1, \ldots, i_d) \in \mathbb{Z}^d_+  : \pi_{i_1, \ldots, i_d} > 0\}.$$  
Note that $\mathrm{sh}(\pi)$ is a diagram of some $(d-1)$-dimensional partition. 
Let $$\mathcal{M}(n_1,\ldots, n_d) = \{ (a_{\mathbf{i}})_{} : a_{\mathbf{i}} \in \mathbb{N}, \mathbf{i} \in [n_1] \times \cdots \times [n_d] 
\}$$ 
be the set of $[n_1] \times \cdots \times [n_d]$ $\mathbb{N}$-matrices 
and  
$$\mathcal{P}(n_1,\ldots, n_{d+1}) := \{ \pi \in \mathcal{P}^{(d)} : D(\pi) \subseteq [n_1] \times \cdots \times [n_{d+1}]\}$$
be the set of {\it boxed} $d$-dimensional partitions. 

For $d = 2,3$ partitions are called {\it plane partitions} and {\it solid partitions}.\footnote{In some literature, there is a +1 shift in dimensions, when partitions are associated with their diagrams. }

\section{A bijection between $d$-dimensional $\mathbb{N}$-matrices and partitions}\label{sec:bij}
\subsection{Last passage matrix} A lattice path in $\mathbb{Z}^d$ is called {\it directed} if it uses only steps of the form $\mathbf{i} \to \mathbf{i} + \mathbf{e}_\ell$ for $\mathbf{i} \in \mathbb{Z}^d$ and $\ell \in [d]$. 
Given a  $d$-dimensional $\mathbb{N}$-matrix $A = (a_{i_1, \ldots, i_d})$, 
define the {\it last passage times} \footnote{ We use terminology related to probabilistic model of last passage percolation, see Sec.~\ref{sec:lpp}.} 
\begin{align*}
G_{i_1, \ldots, i_d} := \max_{\Pi\, :\, (i_1, \ldots, i_d) \to \infty^{d}} \sum_{(j_1, \ldots, j_d) \in \Pi} a_{j_1, \ldots, j_d},
\end{align*}
where the maximum is over directed lattice paths $\Pi$ which start at $(i_1, \ldots, i_d) \in \mathbb{Z}^d_+$. It is easy to see that the following recurrence relation holds
\begin{align}\label{rec}
G_{\mathbf{i}} = a_{\mathbf{i}} + \max_{\ell \in [d]} G_{\mathbf{i} + \mathbf{e}_\ell}, \quad \mathbf{i} \in \mathbb{Z}^d_+.
\end{align}
Notice that the matrix $G  = (G_{\mathbf{i}})_{\mathbf{i} \in \mathbb{Z}^d_+} \in \mathcal{P}^{(d)}$ is a $d$-dimensional partition.

\subsection{The bijection} Define the map $\Phi : \mathcal{M}^{(d)} \to \mathcal{P}^{(d)}$ as follows 
\begin{align}
\Phi : A \longmapsto G
\end{align}
Let $\rho \subset \mathbb{Z}^d_+$ be a shape of some $d$-dimensional partition (or a diagram of a $(d-1)$-dimensional partition). Let 
\begin{align*}
\mathcal{P}(\rho, n) := \{ \pi \in \mathcal{P}^{(d)} : \mathrm{sh}(\pi) \subseteq \rho,\, \pi_{1, \ldots, 1} \le n\}
\end{align*} 
be the set of $d$-dimensional partitions whose shape is a subset of $\rho$ 
and the largest entry 
is at most $n$. 
Let 
\begin{align*}
\mathcal{M}(\rho, n) := \{A = (a_{\mathbf{i}}) \in \mathcal{M}^{(d)} : a_\mathbf{i} > 0 \implies  \mathbf{i} \in \rho,\, G_{1,\ldots, 1} \le n\}
\end{align*}
be the set of $d$-dimensional $\mathbb{N}$-matrices whose support (i.e. the set of indices corresponding to positive entries) lies inside $\rho$ and the largest last passage time is at most $n$.  

\begin{theorem}\label{bij}
The map $\Phi$ defines a bijection between the sets $\mathcal{M}(\rho, n)$ and $\mathcal{P}(\rho, n)$. 
\end{theorem}
\begin{proof}
Let $A = (a_{\mathbf{i}}) \in \mathcal{M}(\rho, n)$. By construction of the map, it is not difficult to see that $\pi = \Phi(A) \in \mathcal{P}(\rho, n)$. Indeed, we have the largest last passage time $\pi_{1,\ldots, 1} \le n$, and $\mathrm{sh}(\pi) \subseteq \rho$ since if $a_{\mathbf{i}} > 0$ then $\mathbf{i} \in \rho$. 

Conversely, given $\pi \in \mathcal{P}(\rho, n)$, to reconstruct the inverse map $\Phi^{-1}$, using the recurrence \eqref{rec} we define the matrix $A = (a_{\mathbf{i}})$ given by
\begin{align}\label{rec1}
a_{\mathbf{i}} = \pi_{\mathbf{i}} - \max_{\ell \in [d]} \pi_{\mathbf{i} + \mathbf{e}_\ell} \ge 0, \quad \mathbf{i} \in \mathbb{Z}^d_+.
\end{align}
Let $G = (G_{\mathbf{i}}) = \Phi(A)$. Let us check that $G = \pi$ and $A \in \mathcal{M}(\rho, n)$. Since $\mathrm{sh}(\pi) \subseteq \rho$ we have $a_{\mathbf{i}} = 0$ for all $\mathbf{i} \not\in \rho$ (in particular, $A \in \mathcal{M}(\rho, \infty)$). Hence $G_{\mathbf{i}} = \pi_{\mathbf{i}} = 0$ for all $\mathbf{i} \not\in \rho$. Consider the directed graph $\Gamma$ on the vertex set $\mathbf{i} \in \rho$ and edges $\mathbf{i} \to \mathbf{i} + \mathbf{e}_{\ell}$ (when $\mathbf{i} + \mathbf{e}_{\ell} \in \rho$) for $\ell \in [d]$. Then $\Gamma$ is acyclic (i.e. has no directed cycles). Notice that $a_{\mathbf{i}} = \pi_{\mathbf{i}} = G_{\mathbf{i}}$ if a vertex $\mathbf{i} \in \Gamma$ has no outgoing edges. Since $\Gamma$ is acyclic, we can sort its vertices in linear order $(\mathbf{i}^{(1)}, \ldots, \mathbf{i}^{(m)})$ so that the edges go only in one direction $\mathbf{i}^{(\ell)} \to \mathbf{i}^{(k)}$ for $\ell < k$. We already noticed that $\pi_{\mathbf{i}^{(m)}} = G_{\mathbf{i}^{(m)}}$. Then inductively on $\ell = m-1,\ldots, 1$ we have 
$$
\pi_{\mathbf{i}^{(\ell)}} = a_{\mathbf{i}^{(\ell)}} + \max_{\mathbf{i}^{(\ell)} \to \mathbf{i}^{(k)}} \pi_{\mathbf{i}^{(k)}} = a_{\mathbf{i}^{(\ell)}} + \max_{\mathbf{i}^{(\ell)} \to \mathbf{i}^{(k)}} G_{\mathbf{i}^{(k)}} = G_{\mathbf{i}^{(\ell)}}.
$$
Therefore, $\pi = G$. In particular, $G_{1,\ldots, 1} \le n$ and hence $A \in \mathcal{M}(\rho, n)$. 
\end{proof}

\begin{corollary}
The map $\Phi$ defines a bijection between each of the following pairs of sets: 
\begin{itemize}
\item[(i)] $\mathcal{M}([n_1]\times \cdots \times [n_d], n_{d+1})$  and $\mathcal{P}(n_1,\ldots, n_{d+1})$

\item[(ii)] $\mathcal{M}(n_1, \ldots, n_d)$  and $\mathcal{P}(n_1,\ldots, n_{d}, \infty)$

\item[(iii)] $\mathcal{M}(\rho, \infty)$ and $\mathcal{P}(\rho, \infty)$ 

\item[(iv)] $\mathcal{M}^{(d)}$ and $\mathcal{P}^{(d)}$.
\end{itemize}
\end{corollary}

\begin{remark}
The item (i) above states that the set of boxed $d$-dimensional partitions with diagrams inside the box $[n_1] \times \cdots \times [n_{d+1}]$ is equal to the number of $[n_1]\times \cdots \times [n_d]$ $\mathbb{N}$-matrices whose largest last passage time is at most $n_{d+1}$. 
\end{remark}


\begin{remark}
For $d = 2$, the map $\Phi$ gives a bijection between $\mathbb{N}$-matrices and plane partitions. This bijection is essentially equivalent (up to diagram rotations) to the one studied in \cite{dy4, dy5}. 
Note that one can construct $d$-dimensional partitions $G$ dynamically using an insertion type procedure as in RSK. Note also that similar largest path (last passage time) properties hold for RSK  as well, see \cite{pak, sagan}. 
\end{remark}

\section{Multivariate identities}\label{sec:mac}

\subsection{Corners} Given a partition $\pi \in \mathcal{P}^{(d)}$, define the set of {\it corners} as follows
$$
\mathrm{Cor}(\pi) := \{\mathbf{i} \in \mathbb{Z}^{d+1}_+ : \mathbf{i} \in D(\pi),\, \mathbf{i} + \mathbf{e}_{\ell} \not\in D(\pi) \text{ for all } \ell \in [d] \}.
$$
(Here $\{e_{\ell}\}$ is the standard basis in $\mathbb{Z}^{d+1}$.) Let $\mathrm{cor}(\pi) := |\mathrm{Cor}(\pi)|$ be the number of corners of $\pi$. Define also the set of {\it top corners} as follows
$$
\mathrm{Cr}(\pi) := \{\mathbf{i} \in \mathbb{Z}^{d+1}_+ : \mathbf{i} \in D(\pi),\, \mathbf{i} + \mathbf{e}_{\ell} \not\in D(\pi) \text{ for all } \ell \in [d + 1] \}\, \subseteq\, \mathrm{Cor}(\pi).
$$
Let $\mathrm{cr}(\pi) := |\mathrm{Cr}(\pi)|$ be the number of top corners of $\pi$. Note that the set of corners $\mathrm{Cr}(\pi)$ uniquely determines the partition $\pi$. 

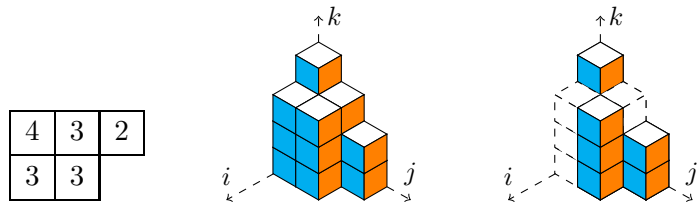
\begin{figure}
\ytableausetup{aligntableaux = bottom}
\begin{ytableau}
 4 & {3} & {2} \\
 {3} & 3  \\  
\end{ytableau}
\qquad
\begin{tikzpicture}[scale = 0.35]
\planepartition{
{4,3,2},
{3,3}}

\draw[->, dashed] (0,4) -- (0,5);
\node at (0.6,5) {\small $k$};

\draw[->, dashed] (2.6,-1.5) -- (3.5,-2);
\node at (3.5, -1) {\small $j$};

\draw[->, dashed] (-1.7,-1.0) -- (-3.5,-2);
\node at (-3.5, -1.1) {\small $i$};

\end{tikzpicture}
\qquad
\begin{tikzpicture}[scale = 0.35]

\planepartition{
{4,0,2},
{0,3}}

\pplanepartition{
{3,3,0},
{0,0}}

\planepartition{
{4,0,2},
{0,3}}

\pplanepartition{
{0,3,0},
{3,0}}

\planepartition{
{0,0,2},
{0,3}}

\draw[->, dashed] (0,4) -- (0,5);
\node at (0.6,5) {\small $k$};

\draw[->, dashed] (2.6,-1.5) -- (3.5,-2);
\node at (3.5, -1) {\small $j$};

\draw[->, dashed] (-1.7,-1.0) -- (-3.5,-2);
\node at (-3.5, -1.1) {\small $i$};

\end{tikzpicture}
\caption{A plane partition $\pi \in \mathcal{P}^{(2)}$ whose $\mathrm{sh}(\pi)$ corresponds to the partition $(3,2)$; its boxed diagram presentation as a pile of cubes in $\mathbb{R}^3$; and boxes of this diagram which correspond to corners. 
}\label{fig0}
\end{figure}

\begin{example}
Let $d = 2$ and $\pi$ be the plane partition given in Fig.~\ref{fig0}. We then have 
\begin{align*}
\mathrm{Cor}(\pi) &= \{(i,j,k) \in D(\pi) : (i+1,j,k), (i,j+1,k) \not\in D(\pi) \} \\
&= \{(1,1,4), (1,3,1), (1,3,2), (2,2,1), (2,2,2), (2,2,3) \}\\
\mathrm{Cr}(\pi) &= \{(i,j,k) \in D(\pi) : (i+1,j,k), (i,j+1,k),(i,j,k+1) \not\in D(\pi) \} \\
&= \{(1,1,4), (1,3,2), (2,2,3)\}
\end{align*}
where corners in Fig.~\ref{fig0} correspond to local configurations 
\begin{tikzpicture}[scale = 0.2]
\cccube{0}{0}{0}
\end{tikzpicture}
and top corners correspond to the configurations
\begin{tikzpicture}[scale = 0.2]
\planepartition{
{1}}
\end{tikzpicture}.
\end{example}


\subsection{Main formulas} For each $i \in [d]$, let $\mathbf{x}^{(i)} = (x^{(i)}_1, x^{(i)}_2, \ldots)$ be a countable set of indeterminate variables. 

\begin{theorem}\label{mult} 
Let $\rho \subset \mathbb{Z}^d_+$ be a fixed shape of a $d$-dimensional partition.  
We have the following multivariate generating function identities
{
\begin{align}
\sum_{{\pi \in \mathcal{P}^{(d)}, \atop \, \mathrm{sh}(\pi) \subseteq \rho} }\, \prod_{(i_1,\ldots, i_{d+1}) \in \mathrm{Cor}(\pi)} x^{(1)}_{i_1} \cdots x^{(d)}_{i_d} &= \prod_{(i_1,\ldots, i_d) \in \rho} {\left(1 - x^{(1)}_{i_1} \cdots x^{(d)}_{i_d} \right)^{-1}} \label{m1}\\
\sum_{ {\pi \in \mathcal{P}^{(d)},\atop \, \mathrm{sh}(\pi) = \rho} }\, \prod_{(i_1,\ldots, i_{d+1}) \in \mathrm{Cor}(\pi)} x^{(1)}_{i_1} \cdots x^{(d)}_{i_d} &=  \prod_{(i_1,\ldots, i_d) \in \mathrm{Cr}(\rho)} x^{(1)}_{i_1} \cdots x^{(d)}_{i_d} \prod_{(i_1,\ldots, i_d) \in \rho} {\left(1 - x^{(1)}_{i_1} \cdots x^{(d)}_{i_d}\right)^{-1}} 
 \label{m2}
\end{align}
}
\end{theorem}

It is convenient to define weights of matrices and partitions as follows. Given a matrix $A = (a_{i_1,\ldots,i_d})\in \mathcal{M}^{(d)}$, we associate to it a multivariable monomial weight 
$$w_A := \prod_{(i_1,\ldots,i_d) \in \mathbb{Z}^d_+} \left(x^{(1)}_{i_1} \cdots x^{(d)}_{i_d} \right)^{a_{i_1,\ldots,i_d}}.$$
Given a partition $\pi \in \mathcal{P}^{(d)}$, we associate to it a multivariable monomial weight
$$
w(\pi) := \prod_{(i_1,\ldots, i_{d+1}) \in \mathrm{Cor}(\pi)} x^{(1)}_{i_1} \cdots x^{(d)}_{i_d}.
$$
\begin{lemma}\label{wei}
Let $A = (a_{\mathbf{i}}) \in \mathcal{M}^{(d)}$ and $\pi = (\pi_{\mathbf{i}}) = \Phi(A) \in \mathcal{P}^{(d)}$. Then $w_A = w(\pi)$.
\end{lemma}
\begin{proof}
Let us first show that 
$$
\pi_{\mathbf{i}} - \max_{\ell \in [d]} \pi_{\mathbf{i} + \mathbf{e}_{\ell}} = |\{i_{d+1} : (\mathbf{i}, i_{d+1}) \in \mathrm{Cor}(\pi) \}|, \quad \mathbf{i} \in \mathbb{Z}^{d}_+.
$$
Indeed, $(\mathbf{i}, i_{d+1}) \in \mathrm{Cor}(\pi)$ iff $i_{d+1} > \pi_{\mathbf{i} + \mathbf{e}_\ell}$ for all $\ell \in [d]$. 
From the description of $\Phi$ we then 
have the following equalities
$$a_{\mathbf{i}} = \pi_{\mathbf{i}} - \max_{\ell \in [d]} \pi_{\mathbf{i} + \mathbf{e}_{\ell}} = |\{i_{d+1} : (\mathbf{i}, i_{d+1}) \in \mathrm{Cor}(\pi) \}|, \quad \mathbf{i} = (i_1,\ldots,i_d) \in \mathbb{Z}^d_+.$$
Now we have 
\begin{align*}
w_A &= \prod_{(i_1,\ldots,i_d) \in \mathbb{Z}^d_+} \left(x^{(1)}_{i_1} \cdots x^{(d)}_{i_d} \right)^{a_{i_1,\ldots,i_d}} \\
&=\prod_{(i_1,\ldots,i_d) \in \mathbb{Z}^d_+} \left(x^{(1)}_{i_1} \cdots x^{(d)}_{i_d} \right)^{|\{i_{d+1} : (i_1,\ldots,i_d, i_{d+1}) \in \mathrm{Cor}(\pi) \}|} \\
&=\prod_{(i_1,\ldots, i_{d+1}) \in \mathrm{Cor}(\pi)} x^{(1)}_{i_1} \cdots x^{(d)}_{i_d} \\
&= w(\pi)
\end{align*}
which gives the needed. 
\end{proof}

\begin{lemma}\label{co}
Let $A = (a_{\mathbf{i}}) \in \mathcal{M}^{}(\rho, \infty)$ and $\pi = (\pi_{\mathbf{i}}) = \Phi(A) \in \mathcal{P}^{}(\rho,\infty)$. The following are equivalent:
\begin{itemize}
\item[(a)] $a_{\mathbf{i}} > 0$ for all $\mathbf{i} \in \mathrm{Cr}(\rho)$

\item[(b)] $\mathrm{sh}(\pi) = \rho$.
\end{itemize}
\end{lemma}
\begin{proof}
Let $\mathbf{i} \in \mathrm{Cr}(\rho)$. 
Assume (a) holds. Since $A \in \mathcal{M}^{}(\rho, \infty)$ we have $a_{\mathbf{i} + \mathbf{e}_{\ell}} = 0$ for all $\ell \in [d]$.  Therefore, $\pi_{\mathbf{i}} = a_{\mathbf{i}} > 0$ and $\pi_{\mathbf{i} + \mathbf{e}_{\ell}} = 0$. Hence $\mathrm{sh}(\pi) = \rho$. 

Assume (b) holds. 
Then we have $\pi_{\mathbf{i} + \mathbf{e}_{\ell}} = 0$ for all $\ell \in [d]$. Therefore, $a_{\mathbf{i}} = \pi_{\mathbf{i}} > 0$.
\end{proof}

\begin{proof}[Proof of Theorem~\ref{mult}]
Firstly note that
\begin{align*}
\sum_{A = (a_{\mathbf{i}}) \in \mathcal{M}(\rho, \infty)} w_A
&= \sum_{A = (a_{\mathbf{i}}) \in \mathcal{M}(\rho, \infty)} \prod_{(i_1,\ldots,i_d) \in \rho} \left(x^{(1)}_{i_1} \cdots x^{(d)}_{i_d} \right)^{a_{i_1,\ldots,i_d}} \\
&= \prod_{(i_1,\ldots,i_d) \in \rho} \left(1 - x^{(1)}_{i_1} \cdots x^{(d)}_{i_d} \right)^{-1}.
\end{align*}
On the other hand, using Theorem~\ref{bij} and Lemma~\ref{wei} we have
$$
\sum_{A = (a_{\mathbf{i}}) \in \mathcal{M}(\rho, \infty)} w_A = \sum_{\pi \in \mathcal{P}(\rho, \infty)} w(\pi) = \sum_{\pi \in \mathcal{P}^{(d)},\, \mathrm{sh}(\pi) \subseteq \rho}\, \prod_{(i_1,\ldots, i_{d+1}) \in \mathrm{Cor}(\pi)} x^{(1)}_{i_1} \cdots x^{(d)}_{i_d}
$$ 
and hence the identity \eqref{m1} follows. 

Let $\overline{\mathcal{M}}(\rho, \infty) = \{A \in \mathcal{M}(\rho, \infty) : \mathbf{i} \in \mathrm{Cr}(\rho) \implies a_{\mathbf{i}} > 0\}$. 
Similarly, note that 
\begin{align*}
\sum_{A = (a_{\mathbf{i}}) \in \overline{\mathcal{M}}(\rho, \infty)} w_A
&= \sum_{A = (a_{\mathbf{i}}) \in \mathcal{M}(\rho, \infty)} \prod_{(i_1,\ldots, i_d) \in \mathrm{Cr}(\rho)} x^{(1)}_{i_1} \cdots x^{(d)}_{i_d} \prod_{(i_1,\ldots,i_d) \in \rho} \left(x^{(1)}_{i_1} \cdots x^{(d)}_{i_d} \right)^{a_{i_1,\ldots,i_d}} \\
&=  \prod_{(i_1,\ldots, i_d) \in \mathrm{Cr}(\rho)} x^{(1)}_{i_1} \cdots x^{(d)}_{i_d} \prod_{(i_1,\ldots,i_d) \in \rho} \left(1 - x^{(1)}_{i_1} \cdots x^{(d)}_{i_d} \right)^{-1}.
\end{align*}
On the other hand, using Lemma~\ref{co} we have 
$$
\sum_{A = (a_{\mathbf{i}}) \in \overline{\mathcal{M}}(\rho, \infty)} w_A = \sum_{\pi \in \mathcal{P}(\rho, \infty), \mathrm{sh}(\pi) = \rho} w(\pi) = \sum_{\pi \in \mathcal{P}^{(d)},\, \mathrm{sh}(\pi) = \rho}\, \prod_{(i_1,\ldots, i_{d+1}) \in \mathrm{Cor}(\pi)} x^{(1)}_{i_1} \cdots x^{(d)}_{i_d}
$$
and hence the identity \eqref{m2} follows.
\end{proof}

\subsection{Some special cases} Let us list few immediate special cases of the above formulas.

\begin{corollary}[Boxed case]\label{mboxed}
We have 
\begin{align*}
\sum_{\pi \in \mathcal{P}^{}(n_1,\ldots,n_d, \infty)}\, \prod_{(i_1,\ldots, i_{d+1}) \in \mathrm{Cor}(\pi)} x^{(1)}_{i_1} \cdots x^{(d)}_{i_d} &= \prod_{i_1 = 1}^{n_1} \cdots \prod_{i_d = 1}^{n_d} 
{\left(1 - x^{(1)}_{i_1} \cdots x^{(d)}_{i_d} \right)^{-1}}
\end{align*}
\end{corollary}



\begin{corollary}[Solid partitions, $d = 3$]
Let $\rho$ be a plane partition. We have 
\begin{align*}
\sum_{\pi \in \mathcal{P}^{(3)},\, \mathrm{sh}(\pi) \subseteq D(\rho)}\, \prod_{(i,j,k,\ell) \in \mathrm{Cor}(\pi)} x^{}_{i} y_{j} z_{k} &= \prod_{(i,j,k) \in D(\rho)} {\left(1 - x^{}_{i} y_j z_k \right)^{-1}}\\
\sum_{\pi \in \mathcal{P}^{(3)},\, \mathrm{sh}(\pi) = D(\rho)}\, \prod_{(i,j,k,\ell) \in \mathrm{Cor}(\pi)} x^{}_{i} y_{j} z_{k} &= \prod_{(i,j,k) \in D(\rho)} {\left(1 - x^{}_{i} y_j z_k \right)^{-1}} \prod_{(i, j, k) \in \mathrm{Cr}(\rho)} x^{}_{i} y_j z_k.
\end{align*}
\end{corollary}

\begin{corollary}[Plane partitions, $d = 2$]
Let $\lambda$ be a partition. 
We have 
\begin{align*}
\sum_{\pi \in \mathcal{P}^{(2)},\, \mathrm{sh}(\pi) \subseteq \lambda}\, \prod_{(i,j,k) \in \mathrm{Cor}(\pi)} x^{}_{i} y_{j} &= \prod_{(i,j) \in D(\lambda)} {\left(1 - x^{}_{i} y_j \right)^{-1}}\\
\sum_{\pi \in \mathcal{P}^{(2)},\, \mathrm{sh}(\pi)= \lambda}\, \prod_{(i,j,k) \in \mathrm{Cor}(\pi)} x^{}_{i} y_{j} &= \prod_{(i,j) \in D(\lambda)} {\left(1 - x^{}_{i} y_j  \right)^{-1}} \prod_{(i, j) \in \mathrm{Cr}(\lambda)} x^{}_{i} y_j.
\end{align*}
\end{corollary}

\begin{remark}
For $d = 2$, the formula in the special rectangular case (up to rotation of diagrams of plane partitions) was proved in \cite{dy5}. 
\end{remark}


\section{MacMahon's numbers and statistics}\label{sec:mac1}
\subsection{Corner-hook volume} Let $\pi \in \mathcal{P}^{(d)}$ be a $d$-dimensional partition. For each point $(i_1,\ldots,i_d)$, 
define the {\it cohook} length
$$
\mathrm{ch}(i_1,\ldots,i_d) := i_1 + \ldots + i_d - d + 1.
$$
Define now the {\it corner-hook volume} statistics $|\cdot|_{ch} : \mathcal{P}^{(d)} \to \mathbb{N}$ computed as follows
$$
|\pi|_{ch} := \sum_{(\mathbf{i}, i_{d+1})\, \in\, \mathrm{Cor}(\pi)} \mathrm{ch}(\mathbf{i}).
$$
\begin{example}
Let $d = 2$ and $\pi$ be the plane partition given in Fig.~\ref{fig0}. Recall that 
\begin{align*}
\mathrm{Cor}(\pi) &= \{(i,j,k) \in D(\pi) : (i+1,j,k), (i,j+1,k) \not\in D(\pi) \} \\
&= \{(1,1,4), (1,3,1), (1,3,2), (2,2,1), (2,2,2), (2,2,3) \}
\end{align*}
and hence we have  
$$
|\pi|_{ch} = (1 + 1 - 1) + (1 + 3 - 1) + (1 + 3 - 1) + (2 + 2 - 1) + (2 + 2 - 1) + (2 + 2 - 1) = 16.
$$
\end{example}

\begin{theorem}\label{shaped}
Let $\rho \subset \mathbb{Z}^d_+$ be a fixed shape of a $d$-dimensional partition.  
We have the following generating functions 
\begin{align*}
\sum_{\pi \in \mathcal{P}^{(d)},\, \mathrm{sh}(\pi) \subseteq \rho} t^{\mathrm{cor}(\pi)} q^{|\pi|_{ch}} &= \prod_{(i_1,\ldots,i_d) \in \rho} {\left( 1 - t q^{i_1 + \cdots + i_d - d + 1} \right)^{-1}},\\
\sum_{\pi \in \mathcal{P}^{(d)},\, \mathrm{sh}(\pi) = \rho} t^{\mathrm{cor}(\pi)} q^{|\pi|_{ch}} &= t^{\mathrm{cr}(\rho)} q^{|\rho|_{cr}} \prod_{(i_1,\ldots,i_d) \in \rho} {\left( 1 - t q^{i_1 + \cdots + i_d - d + 1} \right)^{-1}}, 
\end{align*}
where 
$$|\rho|_{cr} := \sum_{(i_1,\ldots, i_d) \in \mathrm{Cr}(\rho)} \mathrm{ch}(i_1,\ldots,i_d). 
$$
\end{theorem}

\begin{proof}
In Theorem~\ref{mult} set $x^{(1)}_{i} = t q^{i}$ and $x^{(k)}_{i} = q^{i-1}$ for all $i \ge 1$ and $k \ge 2$. 
\end{proof}

\begin{corollary}[Boxed version] We have
\begin{align*}
\sum_{\pi \in \mathcal{P}^{}(n_1,\ldots, n_d, \infty)} t^{\mathrm{cor}(\pi)} q^{|\pi|_{ch}} &= 
\prod_{i_1 = 1}^{n_1} \cdots \prod_{i_d = 1}^{n_d}
{\left( 1 - t q^{i_1 + \cdots + i_d - d + 1} \right)^{-1}}.
\end{align*}
\end{corollary}

\begin{corollary}[Full generating function] \label{full}
We have
$$
\sum_{\pi \in \mathcal{P}^{(d)}} t^{\mathrm{cor}(\pi)} q^{|\pi|_{ch}} = \prod_{n \ge 1} (1 - t q^{n})^{-\binom{n + d - 2}{d - 1}}.
$$
\end{corollary}
\begin{corollary}[Interpretation of MacMahon's numbers]\label{intm}
We have 
$$
\sum_{\pi \in \mathcal{P}^{(d)}} q^{|\pi|_{ch}} = \prod_{n \ge 1} (1 - q^{n})^{-\binom{n + d - 2}{d - 1}}  = \sum_{n = 0}^{\infty} m_d(n) q^n
$$
and hence
$$m_d(n) = |\{ \pi \in \mathcal{P}^{(d)} : |\pi|_{ch} = n \}|,$$ 
i.e. $m_d(n)$ is the number $d$-dimensional partitions whose corner-hook volume is $n$.
\end{corollary}

\begin{corollary}[Pyramid partitions]
Let $\Delta_d({m})$ be a $d$-dimensional partition whose diagram is 
$D(\Delta_{d}({m})) = \{(i_1,\ldots, i_{d+1}) : \mathbb{Z}^{d+1}_+ : i_1 + \cdots + i_{d+1} - d \le m\}$. We have
$$
\sum_{\pi \in \mathcal{P}^{}(\Delta_{d-1}({m}), \infty)} t^{\mathrm{cor}(\pi)} q^{|\pi|_{ch}} = \prod_{n = 1}^m (1 - t q^{n})^{-\binom{n + d - 2}{d - 1}}.
$$
\end{corollary}

\begin{corollary}[$q = 1$ specialization]
We have 
\begin{align*}
\sum_{\pi \in \mathcal{P}^{(d)},\, \mathrm{sh}(\pi) \subseteq \rho} t^{\mathrm{cor}(\pi)}  &= 
{\left( 1 - t \right)^{-|\rho|}},\\
\sum_{\pi \in \mathcal{P}^{(d)},\, \mathrm{sh}(\pi) = \rho} t^{\mathrm{cor}(\pi)} &= t^{\mathrm{cr}(\rho)} 
{\left( 1 - t \right)^{-|\rho|}}. 
\end{align*}
Then the number of $\pi \in \mathcal{P}^{(d)}$ of shape $\rho$ with $k$ corners is equal to $\binom{k - \mathrm{cr}(\rho) + |\rho| - 1}{|\rho| - 1}$. 
\end{corollary}

\subsection{Solid partitions, $d = 3$} Let us restate some of these results for solid partitions. Let $\pi \in \mathcal{P}^{(3)}$ be a solid partition. We then have 
$$
|\pi|_{ch} = \sum_{(i,j,k,\ell)\in \mathrm{Cor}(\pi)} (i + j + k - 2).
$$
\begin{corollary} Let $\rho$ be a fixed plane partition. We have
\begin{align*}
\sum_{\pi \in \mathcal{P}^{(3)},\, \mathrm{sh}(\pi) \subseteq \rho} t^{\mathrm{cor}(\pi)} q^{|\pi|_{ch}} &= \prod_{(i, j, k) \in D(\rho)} {\left( 1 - t q^{i + j + k - 2} \right)^{-1}}\\
\sum_{\pi \in \mathcal{P}^{(3)},\, \mathrm{sh}(\pi) = \rho} t^{\mathrm{cor}(\pi)} q^{|\pi|_{ch}} &= t^{\mathrm{cr}(\rho)} q^{|\rho|_{cr}} \prod_{(i, j, k) \in D(\rho)} {\left( 1 - t q^{i + j + k - 2} \right)^{-1}} 
\end{align*}
and in particular the boxed version
\begin{align*}
\sum_{\pi \in \mathcal{P}^{}(n_1,n_2, n_3, \infty)} t^{\mathrm{cor}(\pi)} q^{|\pi|_{ch}} &= 
\prod_{i = 1}^{n_1} \prod_{j = 1}^{n_2} \prod_{k = 1}^{n_3}
{\left( 1 - t q^{i + j + k - 2} \right)^{-1}}
\end{align*}
\end{corollary}

\subsection{Plane partitions, $d = 2$} Similarly, let us restate some of these results for plane partitions. Let $\pi \in \mathcal{P}^{(2)}$ be a plane partition. We then have 
$$
|\pi|_{ch} = \sum_{(i,j,k)\in \mathrm{Cor}(\pi)} (i + j - 1).
$$
\begin{corollary} Let $\lambda$ be a fixed partition. We have
\begin{align*}
\sum_{\pi \in \mathcal{P}^{(2)},\, \mathrm{sh}(\pi) \subseteq \lambda} t^{\mathrm{cor}(\pi)} q^{|\pi|_{ch}} &= \prod_{(i, j) \in D(\lambda)} {\left( 1 - t q^{i + j - 1} \right)^{-1}}\\
\sum_{\pi \in \mathcal{P}^{(2)},\, \mathrm{sh}(\pi) = \lambda} t^{\mathrm{cor}(\pi)} q^{|\pi|_{ch}} &= t^{\mathrm{cr}(\lambda)} q^{|\lambda|_{cr}} \prod_{(i, j) \in D(\lambda)} {\left( 1 - t q^{i + j - 1} \right)^{-1}} 
\end{align*}
and in particular the boxed version
\begin{align*}
\sum_{\pi \in \mathcal{P}^{}(n_1,n_2 \infty)} t^{\mathrm{cor}(\pi)} q^{|\pi|_{ch}} &= 
\prod_{i = 1}^{n_1} \prod_{j = 1}^{n_2} 
{\left( 1 - t q^{i + j - 1} \right)^{-1}}
\end{align*}
\end{corollary}
Let us look on the last boxed formula. On the other hand, the following {\it trace generating function} is known for plane partitions (see e.g. \cite[Thm~7.20.1]{sta})
$$
\prod_{i = 1}^{n_1} \prod_{j = 1}^{n_2} 
{\left( 1 - t q^{i + j - 1} \right)^{-1}} 
= \sum_{\pi \in \mathcal{P}^{}(n_1,n_2 \infty)} t^{\mathrm{tr}(\pi)} q^{|\pi|},
$$
where $\mathrm{tr}(\pi) := \sum_{i} \pi_{i,i}$ is the trace of a plane partition. 
Therefore, in this case we actually have the following {\it equidistribution} result.
\begin{theorem}[Equidistribution of (tr, vol) and (cor, ch-vol) for plane partitions]
We have 
$$
\sum_{\pi \in \mathcal{P}^{}(n_1,n_2 \infty)} t^{\mathrm{cor}(\pi)} q^{|\pi|_{ch}} = \sum_{\pi \in \mathcal{P}^{}(n_1,n_2 \infty)} t^{\mathrm{tr}(\pi)} q^{|\pi|}.
$$
\end{theorem}
\begin{remark}
Up to a variation of the $|\cdot|_{ch}$ statistic, this result was proved by the second author in \cite{dy5}. 
We also have a direct bijective argument for (a stronger version of) this identity which is somewhat long 
 and will be addressed elsewhere. 
\end{remark} 
\begin{remark}
The formulas in Theorem~\ref{shaped} can be viewed as higher-dimensional analogues of the well-known formula  
$$
\sum_{\mathrm{sh}(\pi) = \lambda} q^{|\pi|} = \prod_{(i,j) \in D(\lambda)} \left(1 - q^{h_{\lambda}(i,j)} \right)^{-1},
$$
where $\lambda$ is a (usual) partition, $h_{\lambda}(i,j) = \lambda_i - i + \lambda'_j - j + 1$ are hook lengths, and the sum runs over {\it reverse plane partitions} $\pi$, see \cite[Ch.~7.22]{sta}.
Its combinatorial proof is known as the {Hillman-Grassl correspondence} \cite{hg}. 
\end{remark}

\begin{remark}
There are various enumeration and generating function formulas known for classes of {\it symmetric} plane partitions, see \cite{sta1}. Similarly, one can define classes of symmetries of diagrams for $d$-dimensional partitions. Are there any explicit corner-hook generating functions over symmetric $d$-dimensional partitions as in Theorem~\ref{shaped}? 
\end{remark}


\subsection{Other statistics}
Theorem~\ref{mult} is a source for many statistics over $d$-dimensional partitions, whose generating functions can be computed explicitly by taking appropriate specializations. For instance, another interesting 
statistic $|\cdot|_{c} : \mathcal{P}^{(d)} \to \mathbb{N}$ is given by  
$$
|\pi|_c := \sum_{(i_1,\ldots, i_{d+1}) \in \mathrm{Cor}(\pi)} i_1, \quad \pi \in \mathcal{P}^{(d)}.
$$
Then via the substitution $x^{(1)}_i \to q^i$ and $x^{(k)}_{i} = 1$ for all $k \ge 2$ and $i \ge 1$ we obtain the following generating function
\begin{align*}
\sum_{\pi \in \mathcal{P}^{}(n_1,\ldots, n_{d}, \infty)} q^{|\pi|_c} = \prod_{i = 1}^{n_1} (1 - q^{i})^{-n_2\cdots n_d}.
\end{align*}
 Another curious statistic is given by 
 $$
 |\pi|_p := \sum_{(i_1,\ldots, i_{d+1}) \in \mathrm{Cor}(\pi)} (i_1 + 2 \, i_2 + \ldots + d \, i_d), \quad \pi \in \mathcal{P}^{(d)}
 $$
for which via the substitution $x^{(k)}_i = q^{ki}$ for all $k,i \ge 1$, we obtain the following generating function
$$
\sum_{\pi \in \mathcal{P}^{(d)}} q^{|\pi|_p} = \prod_{n = 1}^{\infty} (1 - q^{n})^{-p(n, d)},
$$
where $p(n, d)$ is the number of integer partitions of $n$ into $d$ distinct parts. 

\section{$d$-dimensional Grothendieck polynomials}\label{sec:ddg}
\subsection{Definitions}
Let $\pi$ be a $d$-dimensional partition. 
Define the set
$$
\mathrm{sh}_1(\pi) := 
\{(i_2,\ldots, i_{d+1}) : (i_1,\ldots, i_{d+1})  \in D(\pi) \}
$$
which can be viewed as a shape of $\pi$ with respect to the first coordinate. Note that if $\pi \in \mathcal{P}(n_1,\ldots, n_{d+1})$, then $\mathrm{sh}_1(\pi)$ is a diagram of $(d-1)$-dimensional partition from $\mathcal{P}(n_2,\ldots, n_{d+1}).$
Alternatively, $\mathrm{sh}_1(\pi)$ is the diagram of the partition $(\pi_{1,i_2,\ldots,i_d})$. 
For example, if $\pi$ is the plane partition in Fig.~\ref{fig0}, then $\mathrm{sh}_1(\pi)$ corresponds to the partition $(4,3,2)$ which is the first row of $\pi$. 

Throughout this section, let us assume that 
we have the sets of variables
$$\mathbf{x}^{(i)} = (x^{(i)}_1,\ldots, x^{(i)}_{n_i}), \quad i \in [d].$$ 

\begin{definition}
Let $\rho$ be a 
$(d-1)$-dimensional partition from the set $\mathcal{P}(n_2,\ldots,n_{d+1})$. Define the {\it $d$-dimensional Grothendieck polynomials} in $d$ sets of variables 
as follows
\begin{align}\label{defg}
g_{\rho}(\mathbf{x}^{(1)}; \ldots; \mathbf{x}^{(d)}) := \sum_{\pi\, :\, \mathrm{sh}_1(\pi) = \rho}\, \prod_{(i_1,\ldots, i_{d+1}) \in \mathrm{Cor}(\pi)} x^{(1)}_{i_1} \cdots x^{(d)}_{i_d},
\end{align}
where the sum runs over $d$-dimensional partitions $\pi \in \mathcal{P}(n_1,\ldots, n_{d+1})$ with $\mathrm{sh}_1(\pi) = \rho$ (here $\rho$ is identified with its diagram). 
\end{definition}


In the specialization $x^{(k)}_{i} = 1$ for all $k \ge 2$, we simply denote these polynomials by $g_{\rho}(\mathbf{x}) = g_{\rho}(x_1,x_2,\ldots)$ in one set of variables $\mathbf{x}^{(1)} = \mathbf{x} = (x_1,\ldots, x_{n_1})$ so that 
\begin{align}\label{specg}
g_{\rho}(\mathbf{x}^{}) = \sum_{\pi\, :\, \mathrm{sh}_1(\pi) = \rho}\, \prod_{i = 1}^{n_1} x_i^{c_i(\pi)},  \text{ where } c_i(\pi) := |\{\mathbf{i} : (i, \mathbf{i}) \in \mathrm{Cor}(\pi) \}|
\end{align}
and the sum runs over $\pi \in \mathcal{P}(n_1,\ldots, n_{d+1})$. 

\subsection{Examples}
\begin{example}\label{ex:d2}
Consider the case $d = 2$. Let $\lambda \in \mathcal{P}(n_2, n_3)$ be a partition and $\mathbf{x}^{(1)} = \mathbf{x}, \mathbf{x}^{(2)} = \mathbf{y}$. Then \eqref{specg} becomes
$$
g_{\lambda}(\mathbf{x}) = \sum_{\pi\, :\, \mathrm{sh}_1(\pi) = \lambda} \prod_{i = 1}^{n_1} x_i^{c_i(\pi)},   \text{ where } c_i(\pi) = |\{(j,k) : (i, j, k) \in \mathrm{Cor}(\pi) \}|
$$
and the sum runs over plane partitions $\pi \in \mathcal{P}(n_1,n_2,n_3)$. One can see that this gives the dual symmetric Grothendieck polynomials defined in \cite{lp} (but phrased in a slightly different yet equivalent form).\footnote{ The polynomials $\{g_{\lambda}\}$ are usually defined using reverse plane partitions, see \cite{lp, dy, dy2}. } More generally, \eqref{defg} becomes
$$
g_{\lambda}(\mathbf{x}; \mathbf{y}) = \sum_{\pi\, :\, \mathrm{sh}_1(\pi) = \lambda} \prod_{(i,j,k) \in \mathrm{Cor}(\pi)} x_i y_j
$$
which gives a generalized version as in \cite{dy5} or by changing $\tilde g_{\lambda}(\mathbf{x}; \mathbf{y}) = \mathbf{y}^{\lambda} g_{\lambda}(\mathbf{x}; \mathbf{y}^{-1})$ the refined version introduced in \cite{ggl}. These polynomials are symmetric in the variables $\mathbf{x}$. 
\end{example}

\begin{example}\label{notsym}
Let $d = 3$, $(n_1,n_2,n_3,n_4) = (3,2,2,2)$, and $\mathbf{x}^{(1)} = \mathbf{x} = (x_1,x_2,x_3)$, $\mathbf{x}^{(2)} = \mathbf{y} = (y_1,y_2)$, $\mathbf{x}^{(3)} = \mathbf{z} = (z_1,z_2)$. Note that in this case, $3$-dimensional Grothendieck polynomials are  indexed by plane partitions and defined as sums over solid partitions. Consider few examples.

(a) Let $\rho =$
\ytableausetup{aligntableaux = center}
{\scriptsize
\begin{ytableau}
 2 & {1}  
\end{ytableau}
}. Then we have 
{\small
    \begin{align*}
        g_\rho(\mathbf{x};\mathbf{y};\mathbf{z}) = (x_{1}^{2}x_{2} &+ x_{1}^{2}x_{3} + x_{1}x_{2}^{2} + x_{1}x_{3}^{2} + x_{2}^{2}x_{3} + x_{2}x_{3}^{2} + 2 x_{1}x_{2}x_{3}) \cdot y_{1}^{3}z_{1}^{2}z_{2}\\
    	 &+ (x_{1}^{2} + x_{2}^{2} + x_{3}^{2} + x_{1}x_{2} + x_{1}x_{3} + x_{2}x_{3}) \cdot y_{1}^{2}z_{1}z_{2}
    \end{align*}
}
which coincides with the ordinary dual Grothendieck polynomial indexed by the partition $\lambda = (2,1)$, i.e. in this case we have $g_{\lambda}(\mathbf{x}) = g_\rho(\mathbf{x}, \mathbf{1}, \mathbf{1})$.


(b) Let $\rho =$ 
\ytableausetup{aligntableaux = center}
{\scriptsize
\begin{ytableau}
 1 & {1} \\
 {1}  \\  
\end{ytableau}
}. 
Then we have 
{\small
\begin{align*}
g_{\rho}(\mathbf{x}; \mathbf{y}; \mathbf{z}) = 
(x_{1}^{2}x_{2}
	 &+ x_{1}^{2}x_{3}
	 + x_{2}^{2}x_{3}) \cdot y_{1}^{2}y_{2}z_{1}^{2}z_{2}
+ 2 x_{1}x_{2}x_{3} \cdot y_{1}^{2}y_{2}z_{1}^{2}z_{2} \\
& 
+ (x_{1}^{2}
	 + x_{2}^{2}
	 + x_{3}^{2} + 2 x_{1}x_{2}
	 + 2 x_{1}x_{3}
	 + 2 x_{2}x_{3}) \cdot y_{1}y_{2}z_{1}z_{2}
\end{align*}
}
and in particular,
{\small
\begin{align*}
g_{\rho}(\mathbf{x}) = 
x_{1}^{2}x_{2}
	 + x_{1}^{2}x_{3}
	 + x_{2}^{2}x_{3} 
+ 2 x_{1}x_{2}x_{3}
+ x_{1}^{2}
	 + x_{2}^{2}
	 + x_{3}^{2}
+ 2 (x_{1}x_{2}
	 + x_{1}x_{3}
	 + x_{2}x_{3}).
\end{align*}
}

(c) Let $\rho = {\scriptsize\ytableaushort{2 1,1}}$. Then we have 
{\small
    \begin{align*}
        g_{\rho}(\mathbf{x},\mathbf{y},\mathbf{z}) &=
        (3 x_{1}^{2}x_{2}x_{3}
    	 + 3 x_{1}x_{2}^{2}x_{3}
    	 + 2 x_{1}x_{2}x_{3}^{2}
    	 + x_{1}^{2}x_{2}^{2}
    	 + x_{1}^{2}x_{3}^{2}
    	 + x_{2}^{2}x_{3}^{2}
    	 + x_{1}^{3}x_{2}
    	 + x_{1}^{3}x_{3}
    	 + x_{2}^{3}x_{3})\cdot y_{1}^{3}y_{2}z_{1}^{3}z_{2}\\
    	 &+ (4 x_{1}x_{2}x_{3}
    	 + 2 x_{1}^{2}x_{2} + 2 x_{1}^{2}x_{3} + 2 x_{2}^{2}x_{3}
    	 + 3 x_{1}x_{2}^{2} + 3 x_{1}x_{3}^{2} + 3 x_{2}x_{3}^{2}
    	 + x_{1}^{3} + x_{2}^{3} + x_{3}^{3})
    	 \cdot y_{1}^{2}y_{2}z_{1}^{2}z_{2}.
    \end{align*}
}
 
 Let us illustrate few examples of solid partitions contributing to the last expansion.
    \begin{figure}[h]
    \begin{tikzpicture}[scale = 0.35]
        \planepartition{{1,1}};
        \oplanepartition{{2}};
        
        \node at (0,0.5) {\scriptsize $2$};
        \node at (0,1.5) {\scriptsize $1$};
        \node at (0.87,0) {\scriptsize $1$};
        
        \draw[->, dashed] (0,2) -- (0,5);
        \node at (0.6,5) {\small $k$};
        
        \draw[->, dashed] (1.7,-1) -- (3.5,-2);
        \node at (3.5, -1) {\small $j$};
        
        \draw[->, dashed] (-0.85,-0.5) -- (-3.5,-2);
        \node at (-3.5, -1.1) {\small $i$};
        \draw[draw=black,fill=cyan, opacity=1]
               (-0.85,0.5) -- (-0.85,1.5)
            -- (-0.425,1.25) -- (-0.425,0.25) -- cycle;
        \draw[draw=black,fill=orange] (0.85,0.5) -- (0.85,1.5) -- (0.425, 1.25) -- (0.425, 0.25) -- cycle;
        
        
        \node[fill=white] at (0, 3) {\scriptsize $\mathrm{sh}_1(\pi) = \rho$};
    \end{tikzpicture}
    \qquad
    \begin{tikzpicture}[scale = 0.35]
        \oplanepartition{
        {2,1},
        {2,1},
        {1,1}};
        \planepartition{
        {1,1},
        {1,1},
        {1,1}};
        
        \node at (0,0.5) {\scriptsize $2$};
        \node at (0,1.5) {\scriptsize $1$};
        \node at (0.87,0) {\scriptsize $1$};
        
        \node at (-0.87,0) {\scriptsize $2$};
        \node at (-0.87,1) {\scriptsize $1$};
        \node at (0,-0.5) {\scriptsize $1$};
        
        \node at (-1.8,-0.5) {\scriptsize $1$};
        \node at (-0.87,-1) {\scriptsize $1$};
        
        \draw[->, dashed] (0,2) -- (0,5);
        \node at (0.6,5) {\small $k$};
        
        \draw[->, dashed] (1.7,-1) -- (3.5,-2);
        \node at (3.5, -1) {\small $j$};
        
        \draw[->, dashed] (-2.7,-1.5) -- (-3.5,-2);
        \node at (-3.5, -1.1) {\small $i$};

        \draw[draw=black,fill=cyan, opacity=1]
               (-1.7,0) -- (-1.7,1)
            -- (-1.275,0.75) -- (-1.275,-0.25) -- cycle;
        \draw[draw=black,fill=orange] (0.85,0.5) -- (0.85,1.5) -- (0.425, 1.25) -- (0.425, 0.25) -- cycle;
        \node[fill=white] at (0, 3) {\scriptsize $\pi^{(1)}$};
    \end{tikzpicture}
    \qquad
    \begin{tikzpicture}[scale = 0.35]
        \oplanepartition{
        {2,1},
        {2,1},
        {2}};
        \planepartition{
        {1,1},
        {1,1},
        {1}};
        
        \node at (0,0.5) {\scriptsize $2$};
        \node at (0,1.5) {\scriptsize $1$};
        \node at (0.85,0) {\scriptsize $1$};
        
        \node at (-0.85,0) {\scriptsize $2$};
        \node at (-0.85,1) {\scriptsize $1$};
        \node at (0,-0.5) {\scriptsize $1$};
        
        \node at (-1.7,-0.5) {\scriptsize $1$};
        \node at (-1.7,0.5) {\scriptsize $1$};
        
        \draw[->, dashed] (0,2) -- (0,5);
        \node at (0.6,5) {\small $k$};
        
        \draw[->, dashed] (1.7,-1) -- (3.5,-2);
        \node at (3.5, -1) {\small $j$};
        
        \draw[->, dashed] (-2.7,-1.5) -- (-3.5,-2);
        \node at (-3.5, -1.1) {\small $i$};
        
        \draw[draw=black,fill=cyan, opacity=1] (-2.6,-0.5) -- (-2.6,0.5)  -- (-2.175,0.25) -- (-2.175,-0.75) -- cycle;
        \draw[draw=black,fill=orange] (0.85,0.5) -- (0.85,1.5) -- (0.425, 1.25) -- (0.425, 0.25) -- cycle;
        
        
        
        \node[fill=white] at (0, 3) {\scriptsize $\pi^{(2)}$};
    \end{tikzpicture}
    \label{fig1}
    \end{figure}
    Each picture here represents a solid partition as a filling of a diagram of some plane partition with numbers written on top of each box (to make entries of inner boxes visible, some facets are removed). 
    On the left, we have $\mathrm{sh}_{1}(\pi) = \rho$. 
    The next two are solid partitions $\pi^{(1)}$ and $\pi^{(2)}$ represented as fillings of diagrams of plane partitions $\mathrm{sh}(\pi^{(1)}) = \scriptsize\ytableaushort{2 1, 2 1, 1 1}$ and $\mathrm{sh}(\pi^{(2)}) = \scriptsize\ytableaushort{2 1, 2 1, 2}$ ; each has the weight $w(\pi^{(i)})  = x_2^2x_3 \cdot y_1^2 y_2 z_1^2 z_2$; and both have the same $\mathrm{sh}_1(\pi^{(i)}) =\rho$ ($i = 1,2$) displayed on the left. 
\end{example}

\subsection{Properties}

We now prove some properties of $d$-dimensional Grothendieck polynomials.

\begin{theorem}[Cauchy-type identity]
Let $\eta \in \mathcal{P}(n_2,\ldots, n_d)$ be a $(d-2)$-dimensional partition. Let $n \times \eta$ be a $(d-1)$-dimensional partition with the diagram $D(n \times \eta) = \{(i, \mathbf{i}) : i \in [n], \mathbf{i} \in D(\eta) \}$. Then we have the following generating series:
$$
\sum_{\rho \in \mathcal{P}(\eta, \infty)} g_{\rho}(\mathbf{x}^{(1)}; \ldots; \mathbf{x}^{(d)}) 
= \prod_{(i_1,\ldots, i_d) \in D(n_1 \times \eta)} \left(1 - x^{(1)}_{i_1} \cdots x^{(d)}_{i_d} \right)^{-1}.
$$
\end{theorem}

\begin{proof}
Notice that we have
$$
\sum_{\rho \in \mathcal{P}(\eta, \infty)} g_{\rho}(\mathbf{x}^{(1)}; \ldots; \mathbf{x}^{(d)}) = \sum_{\rho \in \mathcal{P}(\eta, \infty)} \sum_{\mathrm{sh}_1(\pi) = \rho} w(\pi) = \sum_{\pi \in \mathcal{P}(n_1 \times \eta, \infty)} w(\pi)
$$
On the other hand, from Theorem~\ref{mult} we have
\begin{align*}
 \sum_{\pi \in \mathcal{P}(n_1 \times \eta, \infty)} w(\pi) = 
 \prod_{(i_1,\ldots, i_d) \in D(n_1 \times \eta)}
 \left(1 - x^{(1)}_{i_1} \cdots x^{(d)}_{i_d} \right)^{-1}
\end{align*}
which gives the result.
\end{proof}

\begin{corollary}\label{gcauchy} 
We have 
$$
\sum_{\rho \in \mathcal{P}(n_2,\ldots, n_d, \infty)} g_{\rho}(\mathbf{x}^{(1)}; \ldots; \mathbf{x}^{(d)}) = \prod_{i_1 = 1}^{n_1} \cdots \prod_{i_d = 1}^{n_d} \left(1 - x^{(1)}_{i_1} \cdots x^{(d)}_{i_d} \right)^{-1}.
$$
\end{corollary}

\begin{lemma}[Simple branching rule]\label{branch}
We have 
$$
g_{\pi}(1, x_1,\ldots, x_n) = \sum_{\rho \subseteq \pi} g_{\rho}(x_1,\ldots, x_n).
$$
\end{lemma}

\begin{proof}
Given a plane partition $\tau$ with $\mathrm{sh}_1(\tau) = \pi$, it contributes to the l.h.s. the weight $\prod_{i = 1}^{n} x_i^{c_{i+1}(\tau)}$ (see eq. \eqref{specg}). Let us form the new partition $\rho \subseteq \pi$ with the diagram 
$$\{(i, \mathbf{i}) : (i+1, \mathbf{i}) \in D(\tau) 
\}$$ 
so that $\prod_{i = 1}^{n} x_i^{c_{i+1}(\tau)} = \prod_{i = 1}^{n} x_i^{c_{i}(\rho)}$ which contributes to the r.h.s. In other words, remove from $D(\tau)$ the points with the first coordinate $1$, then decrease by $1$ the first coordinates for the remaining points. It is not difficult to see that this defines a proper weight-preserving bijection between both sides of the equation.
\end{proof}

Denote $1^k = (1,\ldots, 1)$ with $k$ ones. 

\begin{proposition}[Boxed specialization]\label{boxedspec}
We have 
$$
g_{[n_2] \times \cdots \times [n_{d+1}]}(1^{n_1 + 1}) = g_{[n_2] \times \cdots \times [n_{d+1}]}(1^{n_1 + 1}; 1^{n_2}; \ldots; 1^{n_d}) = |\mathcal{P}(n_1,\ldots, n_{d+1})|.
$$
\end{proposition}
\begin{proof}

Denote $B = [n_2] \times \cdots \times [n_{d+1}]$. Let $\rho$ be a partition diagram inside $B$. From the definition of $g$ we immediately obtain that
$$g_{\rho}(1^{n_1}; \ldots; 1^{n_d}) = |\{\pi \in \mathcal{P}(n_1,\ldots, n_{d+1})\, :\, \mathrm{sh}_1(\pi) = \rho \}|.$$
Therefore, using the branching formula above we get
\begin{align*}
g_{B}(1^{n_1 + 1}; 1^{n_2}; \ldots; 1^{n_d}) 
&= \sum_{\rho \subseteq B} g_{\rho}(1^{n_1}; 1^{n_2}; \ldots; 1^{n_d}) \\
&= \sum_{\rho \subseteq B} |\{\pi \in \mathcal{P}(n_1,\ldots, n_{d+1})\, :\, \mathrm{sh}_1(\pi) = \rho \}| \\
&=|\mathcal{P}(n_1,\ldots, n_{d+1})|
\end{align*}
which gives the needed. 
\end{proof}

\subsection{Quasisymmetry} It is known that the dual Grothendieck polynomials $g_{\lambda}(\mathbf{x})$ are symmetric in $\mathbf{x}$ (in the case $d = 2$). As Example~\ref{notsym} shows, the generalized polynomials $g_{\rho}$ are not necessarily symmetric for $d \ge 3$. However, as we show in this subsection, these polynomials are always quasisymmetric. 

\begin{definition}
A polynomial $f \in \mathbb{Z}[x_1,\ldots, x_n]$ is called {\it quasisymmetric} if for all $1 \le \ell_1 < \cdots < \ell_k \le n$, $1 \le j_1 < \cdots < j_k \le n$, and $a_1,\ldots, a_k \in \mathbb{Z}_+$ we have 
$$
[x^{a_1}_{\ell_1} \cdots x^{a_{k}}_{\ell_k}]\, f = [x^{a_1}_{j_1} \cdots x^{a_{k}}_{j_k}]\, f,
$$
where $[\mathbf{x}^{\alpha}] f$ denotes the coefficient of the monomial $\mathbf{x}^{\alpha}$ in $f$. 
\end{definition}

\begin{theorem}\label{qsym}
We have: $g_{\rho}(\mathbf{x}^{(1)}; \ldots ;\mathbf{x}^{(d)})$ is quasisymmetric in the variables $\mathbf{x}^{(1)}$. 
\end{theorem}
\begin{proof}
To simplify notation let us denote $\mathbf{x}^{(1)} = \mathbf{x} = (x_1, x_2,\ldots)$. We need to show that for all $a_1,\ldots, a_k \in \mathbb{Z}_+$, $1 \le \ell_1 < \cdots < \ell_k \le n_1$, $1 \le j_1 < \cdots < j_k \le n_1$ we have 
$$
[x^{a_1}_{\ell_1} \cdots x^{a_{k}}_{\ell_k}]\, g_{\rho} = [x^{a_1}_{j_1} \cdots x^{a_{k}}_{j_k}]\, g_{\rho}.
$$
Let $L$ and $R$ be the sets of $d$-dimensional partitions which contribute to the l.h.s. and r.h.s. respectively. We are going to construct a weight-preserving bijection $\phi : L \to R$. 

Let $\pi \in L$ for which we have $\pi \in \mathcal{P}(n_1,\ldots,n_{d+1})$ with $\mathrm{sh}_1(\pi) = D(\rho)$ and $w(\pi) = x^{a_1}_{\ell_1} \cdots x^{a_{k}}_{\ell_k} \times w'$, where $w'$ is the the remaining product which does not contain the variables $\mathbf{x}$. 

For a matrix $X = ({x}_{\mathbf{i}})_{\mathbf{i} \in \mathbb{Z}^{d}_+}$, define the submatrices $X^{(\ell)} = (x_{\ell,\mathbf{i}})_{\mathbf{i} \in \mathbf{Z}^{d-1}_+}$. Let $|X|$ denotes the sum of the entries of $X$. 

Let $A = (a_{\mathbf{i}}) = \Phi^{-1}(\pi) \in \mathcal{M}(n_1,\ldots, n_d)$. Note that $A^{(\ell)}  
\in \mathcal{M}(n_2,\ldots, n_d)$ for $\ell \in [n_1]$. Since $\Phi$ preserves weights, i.e. $w_A = w(\pi)$ (see Lemma~\ref{wei}) we must have $A^{(\ell)} \ne \mathbf{0}$ iff $\ell \in \{\ell_1,\ldots, \ell_k \}$. We then have 
$$
w(\pi) = w_A = \prod_{i = 1}^{n_1} x_i^{|A^{(i)}|} \prod_{\mathbf{i} = (i_2,\ldots, i_d)} (x^{(2)}_{i_2} \cdots x^{(d)}_{i_d})^{a_{i,\mathbf{i}}} = \prod_{i = 1}^{k} x_{\ell_i}^{a_i} \times w'.
$$

Let us now construct another matrix $B = (b_{\mathbf{i}}) \in \mathcal{M}(n_1,\ldots, n_d)$ so that $B^{(j)} \ne \mathbf{0}$ iff $j \in \{j_1,\ldots,j_k\}$ and $B^{(j_i)} = A^{(\ell_i)}$ for all $i \in [k]$. Let $\pi' = \Phi(B)$. We then clearly have 
$$
w(\pi') = w_B = \prod_{i = 1}^{k} x_{j_i}^{a_i} \times w'.
$$

Let us show that $\mathrm{sh}_1(\pi') = D(\rho) = \mathrm{sh}_1(\pi)$. Recall that $\mathrm{sh}_1(\pi')$ is the diagram of the partition $(\pi'_{1,i_2,\ldots,i_d})$. By definition of $\Phi$, each entry $\pi'_{1,i_2,\ldots,i_d} $ is the largest weight directed path from $(1,i_2,\ldots, i_d)$ to $(n_1,\ldots, n_d)$ through the matrix $B$. Similarly, each entry $\pi_{1,i_2,\ldots,i_d} $ is the largest weight directed path from $(1,i_2,\ldots, i_d)$ to $(n_1,\ldots, n_d)$ through the matrix $A$. We then have 
\begin{align*}
\pi_{1,i_2,\ldots,i_d}  &= \max_{\Pi : (1,i_2,\ldots, i_d) \to (n_1,\ldots, n_d)} \sum_{(\ell,\mathbf{i}) \in \Pi, \, \ell \in \{\ell_1,\ldots, \ell_k \}} a_{\mathbf{j}} \\
&= \max_{\Pi : (1,i_2,\ldots, i_d) \to (n_1,\ldots, n_d)} \sum_{(j,\mathbf{i}) \in \Pi, \, j \in \{j_1,\ldots, j_k \}} b_{\mathbf{j}} \\
&= \pi'_{1,i_2,\ldots,i_d}.
\end{align*}
Hence $\pi' \in R$, we can set $\phi : \pi \mapsto \pi'$ and 
it is a well-defined bijection between $L$ and $R$.
\end{proof}

Let us define the {\it boxed} polynomials
$$
F_{(n_1,\ldots, n_{d+ 1})}(\mathbf{x}^{(1)}; \ldots; \mathbf{x}^{(d)}) := \sum_{\pi \in \mathcal{P}(n_1,\ldots, n_{d+1})} \prod_{(i_1,\ldots, i_{d+1}) \in \mathrm{Cor}(\pi)} x^{(1)}_{i_1} \cdots x^{(d)}_{i_d}, 
$$
which is a bounded version of the Cauchy product as by Theorem~\ref{mult} we have 
$$
\lim_{n_{d+1} \to \infty} F_{(n_1,\ldots, n_{d+ 1})}(\mathbf{x}^{(1)}; \ldots; \mathbf{x}^{(d)}) =  \prod_{i_1 = 1}^{n_1} \cdots \prod_{i_d = 1}^{n_d} \left(1 - x^{(1)}_{i_1} \cdots x^{(d)}_{i_d} \right)^{-1}.
$$
These polynomials can also be expanded as follows:
$$
F_{(n_1,\ldots, n_{d+ 1})}(\mathbf{x}^{(1)}; \ldots; \mathbf{x}^{(d)}) = \sum_{\rho \in \mathcal{P}(n_2,\ldots,n_{d+1})} \sum_{\mathrm{sh}_{1}(\pi) = \rho} w(\pi) = \sum_{\rho \in \mathcal{P}(n_2,\ldots,n_{d+1})} g_{\rho}(\mathbf{x}^{(1)}; \ldots; \mathbf{x}^{(d)}).
$$
\begin{corollary}[Full quasisymmetry of boxed polynomials]
We have: 
$
F_{(n_1,\ldots, n_{d+ 1})} 
$
is quasisymmetric in each set of the variables $\mathbf{x}^{(1)}; \ldots; \mathbf{x}^{(d)}$ independently. 
\end{corollary}
\begin{proof}
The quasisymmetry in $\mathbf{x}^{(1)}$ is immediate from the previous theorem. The same holds for any other set of variables by noting that the definitions of $\mathrm{Cor}(\pi)$ and weights $\pi$ are symmetric in the first $d$ coordinates and hence we may repeat the proof by `rotation', i.e. moving any coordinate as the first one. 
\end{proof}

\begin{definition}
Let $A = (a_{i_1,\ldots,i_d})\in \mathcal{M}^{}(n_1,\ldots, n_d)$. For each $\ell \in [d]$, consider the matrices $B^{(\ell)}_{i} = (a_{i_1,\ldots,i_d})_{i_{\ell} = i}$, i.e. submatrices of $A$ with fixed $\ell$-th coordinate. Define the vectors 
$$s_{\ell}(A) := (|B^{(\ell)}_1|, |B^{(\ell)}_2|, \ldots),$$ where $|B|$ denotes the sum of entries of $B$. For example, if $d = 2$, then $s_{1}(A)$ is the vector of row sums of $A$, and $s_{2}(A)$ is the column sums of $A$. 
Let us also say that $A$ is a {\it packed matrix} if for each $\ell \in [d]$, the sequence $s_{\ell}(A)$ does not contain zeros between its positive entries. Denote by $\mathrm{pack}(A)$ the packed matrix formed from $A$ by removing its zero submatrices $B^{(\ell)}_i = \mathbf{0}$. 
\end{definition}

For a composition $\alpha = (\alpha_1,\ldots, \alpha_k) \in \mathbb{Z}^k_+$, recall the {\it monomial quasisymmetric functions} 
$$
M_{\alpha}(\mathbf{x}) := \sum_{i_1 < \ldots < i_k} x^{\alpha_1}_{i_1} \cdots x^{\alpha_k}_{i_k}.
$$
Note that they form a basis of the algebra of quasisymmetric functions. 

It is easy to see that 
$$
F_{(n_1,\ldots,n_d, \infty)} = \sum_{A \in \mathcal{M}(n_1,\ldots,n_d)} w_A = \sum_{\alpha^{(1)}, \ldots, \alpha^{(d)}} m^{}_{\alpha^{(1)},\ldots,\alpha^{(d)}}\, (\mathbf{x}^{(1)})^{\alpha^{(1)}} \cdots (\mathbf{x}^{(d)})^{\alpha^{(d)}}, 
$$
where $m^{}_{\alpha^{(1)},\ldots,\alpha^{(d)}}$ is the number of $A \in \mathcal{M}(n_1,\ldots, n_d)$ with $s_{\ell}(A) = \alpha^{(\ell)} \in \mathbb{N}^{n_{\ell}}$. The following result is a finite boxed version of this expansion.
\begin{theorem}[Monomial basis expansion of boxed polynomials]
We have 
$$
F_{(n_1,\ldots,n_{d+1})} = \sum_{\alpha^{(1)}, \ldots, \alpha^{(d)}} m^{(n_{d+1})}_{\alpha^{(1)},\ldots,\alpha^{(d)}}\, M_{\alpha^{(1)}}(\mathbf{x}^{(1)}) \cdots M_{\alpha^{(d)}}(\mathbf{x}^{(d)}),
$$
where the sum runs over compositions $\alpha^{(1)}, \ldots, \alpha^{(d)}$ such that $|\alpha^{(i)}| = |\alpha^{(j)}|$ for all $i,j$, and the coefficient $m^{(n_{d+1})}_{\alpha^{(1)},\ldots,\alpha^{(d)}}$ is equal to the number of packed matrices $A \in \mathcal{M}([n_1] \times \cdots \times [n_d], n_{d+1})$ such that $s_{\ell}(A) = \alpha^{(\ell)}$ for all $\ell \in [d]$. 
\end{theorem}
\begin{proof}
Let $P \in \mathcal{M}([n_1] \times \cdots \times [n_d], n_{d+1})$ be a packed matrix and let $M(P)$ be the set of matrices $A \in \mathcal{M}([n_1] \times \cdots \times [n_d], n_{d+1})$ such that $\mathrm{pack}(A) = P$. Let $s_{\ell}(P) = \alpha^{(\ell)}$. Then (by an argument as in Theorem~\ref{qsym}) it is not difficult to obtain that we have
$$
\sum_{A \in M(P)} w_A = m^{(n_{d+1})}_{\alpha^{(1)},\ldots,\alpha^{(d)}}\, M_{\alpha^{(1)}}(\mathbf{x}^{(1)}) \cdots M_{\alpha^{(d)}}(\mathbf{x}^{(d)}).
$$
Therefore, we obtain 
\begin{align*}
F_{(n_1,\ldots,n_{d+1})} &= \sum_{\pi \in \mathcal{P}(n_1,\ldots, n_{d+1})} w(\pi) \\
&= \sum_{A \in \mathcal{M}([n_1] \times \cdots \times [n_d], n_{d+1})} w_A \\
&= \sum_{P \text{ packed}} \sum_{A \in M(P)} w_A \\
&=  \sum_{\alpha^{(1)}, \ldots, \alpha^{(d)}} m^{(n_{d+1})}_{\alpha^{(1)},\ldots,\alpha^{(d)}}\, M_{\alpha^{(1)}}(\mathbf{x}^{(1)}) \cdots M_{\alpha^{(d)}}(\mathbf{x}^{(d)})
\end{align*}
as needed. 
\end{proof}
\begin{remark}
For $d = 2$, packed matrices appear in the algebra of matrix quasisymmetric functions, see \cite{dht}.  
\end{remark}

\subsection{Dual Grothendieck polynomials, $d = 2$} 
Recall that in this case (see Example~\ref{ex:d2}), we get the following definition of polynomials $g_{\lambda}(\mathbf{x}; \mathbf{y})$ indexed by partitions $\lambda$. We define 
$$
g_{\lambda}(\mathbf{x}; \mathbf{y}) := \sum_{\pi\, :\, \mathrm{sh}_1(\pi) = \lambda} \prod_{(i,j,k) \in \mathrm{Cor}(\pi)} x_i y_j
$$
where the sum runs over plane partitions $\pi$. 
The polynomials $g_{\lambda}(\mathbf{x}; \mathbf{y})$ are generalizations of dual Grothendieck polynomials which correspond to the specialization $g_{\lambda}(\mathbf{x}) = g_{\lambda}(\mathbf{x}; \mathbf{1})$. In fact, $g_{\lambda}(\mathbf{x}; \mathbf{y})$ is symmetric in $\mathbf{x}$. The Cauchy-type identity in Corollary~\ref{gcauchy} becomes
$$
\sum_{\lambda \in \mathcal{P}(n_2, \infty)} g_{\lambda}(\mathbf{x}; \mathbf{y}) = \prod_{i = 1}^{n_1} \prod_{j = 1}^{n_2} \frac{1}{1 - x_i y_j}, 
$$
which was proved in \cite{dy4, dy5}.
The boxed specialization formula in Proposition~\ref{boxedspec} becomes the following 
$$
g_{[n_2] \times [n_3]}(1^{n_1 + 1}) = |\mathcal{P}(n_1, n_2, n_3)|,
$$
the number of plane partitions inside the box $[n_1] \times [n_2] \times [n_3]$, for which there is also the famous {\it MacMahon boxed product formula}
$$
|\mathcal{P}(n_1, n_2, n_3)| = \prod_{i = 1}^{n_1} \prod_{j = 1}^{n_2} \prod_{k = 1}^{n_3} \frac{i + j + k - 1}{i + j + k - 2}. 
$$
Using determinantal formulas for dual Grothendieck polynomials \cite{dy} we also have the following `coincidence' formula (see \cite{dy4,dy5}) connecting them with the Schur polynomials $\{s_{\lambda} \}$ as follows
$$
g_{[n_2] \times [n_3]}(\mathbf{x}) = s_{[n_2] \times [n_3]}(\mathbf{x}, 1^{n_2 - 1}).
$$

\subsection{3d Grothendieck polynomials, $d = 3$} In this case, we get the following definition of polynomials $g_{\rho}(\mathbf{x}; \mathbf{y}; \mathbf{z})$ indexed by plane partitions $\rho$. We define
\begin{align}\label{g3d}
g_{\rho}(\mathbf{x}; \mathbf{y}; \mathbf{z}) := \sum_{\pi\, :\, \mathrm{sh}_1(\pi) = \rho} \prod_{(i,j,k,\ell) \in \mathrm{Cor}(\pi)} x_{i} y_{j} z_{k},
\end{align}
where the sum runs over solid partitions $\pi \in \mathcal{P}(n_1,n_2,n_3,n_4)$. Note also that if $\rho$ satisfies $D(\rho) = \{(1,i,j) : (i,j) \in D(\lambda)\}$ where $\lambda$ is a partition, we then have $g_{\rho}(\mathbf{x}; \mathbf{y}; \mathbf{1}) = g_{\lambda}(\mathbf{x}; \mathbf{y})$ reduces to the 2d case discussed above. The polynomials $g_{\rho}(\mathbf{x}; \mathbf{y}; \mathbf{z})$ are quasisymmetric in $\mathbf{x}$. Then Cauchy-type identity in Corollary~\ref{gcauchy} becomes
$$
\sum_{\rho \in \mathcal{P}(n_2, n_3, \infty)} g_{\rho}(\mathbf{x}^{}; \mathbf{y}; \mathbf{z}) = \prod_{i = 1}^{n_1} \prod_{j = 1}^{n_2} \prod_{k = 1}^{n_3} \left(1 - x^{}_{i} y_j z_k \right)^{-1}.
$$
The boxed specialization formula becomes the following
$$
g_{[n_2] \times [n_3] \times [n_4]}(1^{n_1 + 1}) = |\mathcal{P}(n_1, n_2, n_3, n_4)|,
$$
the number of solid partitions inside the box $[n_1] \times [n_2] \times [n_3] \times [n_4]$. 

\begin{remark}
[On higher-dimensional Schur polynomials and SSYT] Note that the $d$-dimensional Grothendieck polynomials $g_{\rho}(\mathbf{x})$ are inhomogeneous. It is well known that for $d = 2$ we have $g_{\lambda} = s_{\lambda} + \text{lower degree terms}$. By analogy, the top degree homogeneous component of $g_{\rho}(\mathbf{x})$ denoted by $s_{\rho}(\mathbf{x})$ can be viewed as a higher-dimensional analogue of Schur polynomials. It sums over a subset of $d$-dimensional partitions which are analogous to {\it semistandard Young tableaux} (SSYT) for the case $d = 2$. By Theorem~\ref{qsym}, $\{s_{\rho}\}$ are also quasisymmetric polynomials. Are there any interesting properties of these functions and tableaux? 
\end{remark}

\section{Last passage percolation in $\mathbb{Z}^{d}$}\label{sec:lpp}

In this section we consider a directed last passage percolation model with geometric weights and show its connections with $d$-dimensional Grothendieck polynomials studied in the previous section. 

Let $W = (w_{\mathbf{i}})_{\mathbf{i} \in \mathbb{Z}^d_+}$ be a random matrix with i.i.d. entries $w_{\mathbf{i}}$ which have geometric distribution with parameter $q \in (0,1)$, i.e. 
$$
\mathrm{Prob}(w_{\mathbf{i}} = k) = (1 - q)\, q^k, \quad k \in \mathbb{N}.
$$ 
Define the last passage times as follows
$$
G(\mathbf{i}) = G(\mathbf{1} \to \mathbf{i}) = \max_{\Pi : \mathbf{1} \to \mathbf{i}}\,  \sum_{\mathbf{j} \in \Pi} w_{\mathbf{j}}, \quad \mathbf{i} \in \mathbb{Z}^d_+,
$$
where the maximum is over directed lattice paths $\Pi$ from $(1,\ldots, 1)$ to $\mathbf{i}$. Using Kingman's subadditivity theorem, one can show that there is a deterministic {\it limit shape} $\varphi : \mathbb{R}^{d}_{\ge 0} \to \mathbb{R}_{\ge 0}$ (see \cite{martin}) such that as $n \to \infty$ we have a.s. convergence 
$$
\frac{1}{n} G(\lfloor n\mathbf{x} \rfloor) \to \varphi(\mathbf{x}), \quad \mathbf{x} \in \mathbb{R}^d_{\ge 0}.
$$
The case $d = 2$ is exactly solvable and 
$\varphi(x,y) = (x + y + 2\sqrt{qxy})/(1 - q)$; moreover, the fluctuations around the shape are of order $n^{1/3}$ and tend to the Tracy-Widom distribution  \cite{johansson}. However, much less is known for $d \ge 3$. 

Now we are going to show that $d$-dimensional Grothendieck polynomials naturally appear in distribution formulas for this model.

\begin{theorem}\label{lpp}
Let $n_1,\ldots, n_d \in \mathbb{Z}_+$ and $\rho \in \mathcal{P}(n_2,\ldots, n_{d}, \infty)$ be a $(d-1)$-dimensional partition. Denote $\mathbf{n} = (n_2 + 1,\ldots, n_{d} + 1)$ and $N = n_1 \cdots n_{d}$. We have the following joint distribution formula
\begin{align*}
\mathrm{Prob}\left( G(n_1, \mathbf{n} - \mathbf{i}) = \rho_{\mathbf{i}}\, :\,  \mathbf{i} \in [n_2] \times \cdots \times [n_{d}] \right) = (1 - q)^{N}\, g_{\rho}(\underbrace{q,\ldots, q}_{n_1 \text{ times}}). 
\end{align*}
\end{theorem}
\begin{proof}
Let us flip and truncate the matrix $W$ to get $W' =(w'_{\mathbf{i}}) = (w_{(n_1 + 1,\mathbf{n}) - \mathbf{i}} )_{\mathbf{i} \in [n_1] \times \cdots \times [n_d]}$. 

Let $\pi = (\pi_{\mathbf{i}}) \in \mathcal{P}(n_1,\ldots, n_d, \infty)$ and $(a_{\mathbf{i}}) = \Phi^{-1}(\pi)$. We obtain
$$
\mathrm{Prob}(W' = \Phi^{-1}(\pi)) = \prod_{\mathbf{i} \in [n_1] \times \cdots \times [n_{d}]} \mathrm{Prob}(w'_\mathbf{i} = a_{\mathbf{i}}) = (1-q)^{N} q^{S(\pi)}, 
$$
where $S(\pi) = \sum_{\mathbf{i}} a_{\mathbf{i}}$. Note that from \eqref{specg} we have
\begin{align*}
(1 - q)^{N} g_{\rho}(\underbrace{q,\ldots, q}_{n_1 \text{ times}}) &= (1 - q)^{N} \sum_{\pi\, :\, \mathrm{sh}_1(\pi) = \rho} q^{c_1(\pi) + \ldots + c_{n_1}(\pi)} \\
&= \sum_{\pi\, :\, \mathrm{sh}_1(\pi) = \rho} (1 - q)^{N} q^{S(\pi)} \\
&= \sum_{\pi\, :\, \mathrm{sh}_1(\pi) = \rho} \mathrm{Prob}(W' = \Phi^{-1}(\pi)),
\end{align*}
where the sum runs over $\pi \in \mathcal{P}(n_1,\ldots, n_d, \infty)$.
Observe that we have $\Phi(W') = (G(\mathbf{i}))_{\mathbf{i} \in \mathbb{Z}^d_+}$. Therefore, now we get
\begin{align*}
\mathrm{Prob}\left(G(n_1, \mathbf{n} - \mathbf{i}) = \rho_{\mathbf{i}}\, :\,  \mathbf{i} \in [n_2] \times \cdots \times [n_{d}] \right)
&= \sum_{\pi\, :\, \mathrm{sh}_1(\pi) = \rho} \mathrm{Prob}(\Phi(W') = \pi) \\
&= (1 - q)^{N} g_{\rho}(\underbrace{q,\ldots, q}_{n_1 \text{ times}})
\end{align*}
as needed. 
\end{proof}

\begin{corollary}[Single point distribution formula]
We have
$$
\mathrm{Prob}(G(n_1,\ldots, n_d) \le n) = (1 - q)^{N}\, g_{[n_2] \times \cdots \times [n_d] \times [n]}(1,\underbrace{q,\ldots, q}_{n_1 \text{ times}})
$$
\end{corollary}
\begin{proof}
Follows by combining the theorem with Lemma~\ref{branch}. 
\end{proof}

\begin{corollary}[The case $d = 2$]
Let $\lambda \in \mathcal{P}(n_2, \infty)$ be a partition. We have 
\begin{align*}
\mathrm{Prob}\left(G(n_1, n_2 + 1 - i) = \lambda_i \, :\,  i \in [n_2] \right) = (1 - q)^{n_1 n_2}\, g_{\lambda}(\underbrace{q,\ldots, q}_{n_1 \text{ times}}). 
\end{align*}
\end{corollary}

\begin{remark}
This formula (which shows that dual symmetric Grothendieck polynomials arise naturally in the last passage percolation model) was proved in \cite{dy4} and in more general case with different parameters in \cite{dy6}. Note that in this case we can obtain many determinantal formulas. 
\end{remark}

\begin{remark}
Theorem~\ref{lpp} suggests a probability distribution on the set $\mathcal{P}(n_2,\ldots, n_{d}, \infty)$ of $(d-1)$-dimensional partitions defined as follows: 
$$
\mathrm{Prob}_{g}(\rho) := (1 - q)^{n_1 \cdots n_d}\, g_{\rho}(\underbrace{q,\ldots, q}_{n_1 \text{ times}}), \qquad \rho \in \mathcal{P}(n_2,\ldots, n_{d}, \infty).
$$ 
\end{remark}

\section{Concluding remarks and open questions}\label{sec:final}
\subsection{} 
After defining plane partitions in EC2 \cite[Ch.~7.20]{sta}, Richard Stanley writes:
\begin{quote}
`` ... It now seems obvious to define {\it $r$-dimensional partitions} for any $r \ge 1$. However, almost nothing significant is known for $r \ge 3$."
\end{quote}
Few more remarks and references on the subject can be found in an early survey \cite{sta2} (on the theory of plane partitions). For more recent works, see \cite{mr, bgp, gov, dg}. 

\subsection{Asymptotics} MacMahon's numbers $m_d(n)$ have the following asymptotics  \cite{bgp}
$$
\lim_{n \to \infty} n^{-d/(1 + d)}\log m_d(n) = 
\frac{1 + d}{d} \left(d\, \zeta(1 + d) \right)^{1/(1 + d)},
$$ 
where $\zeta$ is the Riemann zeta function (which is computed based on the explicit formula for the generating function). 
It was conjectured in \cite{bgp}  and (for solid partitions)  in \cite{mr} supported by numerical experiments, that $p_d(n)$, the number of $d$-dimensional partitions of volume (size) $n$, has exactly the same asymptotics. However, later computations reported in \cite{dg} suggest that this is not the case (for $d = 3$) and that $p_3(n)$ is asymptotically larger than $m_3(n)$ (despite the fact that $m_3(n) = p_3(n)$ for $n \le 5$ and $m_3(n) > p_3(n)$ for the next many values of $n$ \cite{atkin, dg}; cf. the sequences A000293, A000294 in \cite{oeis}). See also \cite{ekhad} and a useful resource \cite{boltzmann} for more related data. 
Given our interpretation for $m_d(n)$ (Corollary~\ref{intm}), is it possible to compare them with $p_d(n)$? 

\subsection{$d$-dimensional Grothendieck polynomials} Are there any (algebraic, determinantal) formulas for $d$-dimensional Grothendieck polynomials? They will be important for at least two applications: enumeration of boxed higher-dimensional partitions, and computing distribution formulas (or performing asymptotic analysis) for the last passage percolation problem discussed above. 
Note that for $d = 2$, there are several determinantal formulas (Jacobi-Trudi, bialternant types) known, see \cite{dy, ay}. 

\section*{Acknowledgements}
We are greateful to Askar Dzhumadil'daev, Suresh Govindarajan, and Igor Pak for useful conversations. 



\begin{thebibliography}{abcdefghij}
\bibitem[AY20]{ay}
A. Amanov and D. Yeliussizov, Determinantal formulas for dual Grothendieck polynomials, arXiv:2003.03907, 2020.

\bibitem[And98]{andrews}
G. E. Andrews, The theory of partitions, Vol. 2, Cambridge University Press, 1998.

\bibitem[ABMM67]{atkin}
A. O. L. Atkin, P. Bratley, I. G. Macdonald, and J. K. S. McKay, Some computations for $m$-dimensional partitions, Proc. Cambridge Philos. Soc. {\bf 63} (1967) 1097--1100.

\bibitem[BGP12]{bgp}
S. Balakrishnan, S. Govindarajan, and N. S. Prabhakar, On the asymptotics of higher-dimensional partitions, J. Phys. A {\bf 45} (2012), 055001.

\bibitem[BBS13]{bbs}
K. Behrend, J. Bryan, and B. Szendr\H{o}i, Motivic degree zero Donaldson--Thomas invariants, Invent. Math. {\bf 192} (2013), 111--160.

\bibitem[CK18]{ck}
Y. Cao and M. Kool, Zero-dimensional Donaldson--Thomas invariants of Calabi--Yau 4-folds, Adv. Math. {\bf 338} (2018), 601--648.

\bibitem[DHT02]{dht}
G. Duchamp, F. Hivert, and J.-Y. Thibon, Noncommutative symmetric functions VI: free quasi-symmetric functions and related algebras, International Journal of Algebra and computation {\bf 12.05} (2002), 671--717.

\bibitem[Ekh12]{ekhad}
S. B. Ekhad, The number of $m$-dimensional partitions of eleven and twelve, published electronically at \url{https://sites.math.rutgers.edu/~zeilberg/mamarim/mamarimhtml/mDimPars.html}

\bibitem[GGL16]{ggl}
P. Galashin, D. Grinberg, and G. Liu, Refined dual stable Grothendieck polynomials and generalized Bender-Knuth involutions, Electronic J. Combin. {\bf 23} (2016): 3-14.

\bibitem[HG76]{hg}
A. Hillman and R. Grassl, Reverse plane partitions and tableau hook numbers, J. Combin. Theory Ser. A {\bf 21} (1976), 216--221.

\bibitem[DG15]{dg}
N. Destainville and S. Govindarajan, Estimating the asymptotics of solid partitions, J. Stat. Phys. {\bf 158} (2015), 950--967.

\bibitem[Gov]{boltzmann}
S. Govindarajan, The partitions project \url{http://boltzmann.wikidot.com/the-partitions-project}

\bibitem[Gov13]{gov}
S. Govindarajan, Notes on higher-dimensional partitions, J. Combin. Theory Ser. A {\bf 120} (2013), 600--622.

\bibitem[Joh00]{johansson}
K. Johansson,  Shape fluctuations and random matrices, Commun. Math. Phys. {\bf 209} (2000), 437--476.

\bibitem[Krat16]{krat}
C. Krattenthaler, Plane partitions in the work of Richard Stanley and his school, The mathematical
legacy of Richard P. Stanley (2016), 246--277.

\bibitem[Knu70]{knuth}
D. Knuth, A note on solid partitions,  Math. Comp. {\bf 24} (1970) 955--961.

\bibitem[LP07]{lp}
T.~Lam and P.~Pylyavskyy, 
Combinatorial Hopf algebras and K-homology of Grassmannians,
{Int. Math. Res. Not.}\ Vol. 2007, (2007), rnm 125.

\bibitem[Mac16]{macmahon}
P. A. MacMahon, Combinatory Analysis, Cambridge University Press, Vol. 1 and 2, 1916.

\bibitem[Mar06]{martin}
J. B. Martin, Last-passage percolation with general weight distribution,
Markov Processes and Related Fields {\bf 12}  (2006), 273--299.

\bibitem[MR03]{mr}
V. Mustonen and R. Rajesh, Numerical estimation of the asymptotic behaviour of solid partitions of an integer, J. Phys. A 36 {\bf 24} (2003), 6651--6659.

\bibitem[Nek17]{nekrasov}
N. Nekrasov, Magnificent four, arXiv:1712.08128, 2017.

\bibitem[OEIS]{oeis}
Online Encyclopedia of Integer Sequences \url{https://oeis.org}

\bibitem[Pak01]{pak}
I. Pak, Hook length formula and geometric combinatorics, S\'em. Lothar. Combin. {\bf 46} (2001): 6.

\bibitem[Sag01]{sagan}
B. Sagan, The Symmetric Group, Springer, New York, 2001.

\bibitem[Sta71]{sta2}
R. P. Stanley, Theory and application of plane partitions, Parts 1 and 2, Studies in Applied Math. {\bf 50} (1971), 167--188, 259--279.

\bibitem[Sta86]{sta1}
R. P. Stanley, Symmetries of plane partitions, J. Combin. Theory Ser. A {\bf 43} (1986), 103--113.

\bibitem[Sta99]{sta}
R. P. Stanley, Enumerative Combinatorics, Vol. 2, Cambridge, 1999.

\bibitem[Yel17]{dy}
D. Yeliussizov, Duality and deformations of stable Grothendieck polynomials, J. Algebraic Combin. {\bf 45} (2017), 295--344.

\bibitem[Yel19]{dy2}
D. Yeliussizov, Symmetric Grothendieck polynomials, skew Cauchy identities, and dual filtered Young graphs, J. Combin. Theory Ser. A {\bf 161} (2019), 453--485.

\bibitem[Yel19a]{dy4}
D. Yeliussizov, Random plane partitions and corner distributions, arXiv:1910.13378, 2019.

\bibitem[Yel19b]{dy5}
D. Yeliussizov, Enumeration of plane partitions by descents, arXiv:1911.03259, 2019.

\bibitem[Yel20]{dy6}
D. Yeliussizov, Dual Grothendieck polynomials via last-passage percolation, C. R. Math. Acad. Sci. Paris {\bf 358} (2020), 497--503.

\end{thebibliography}
\end{document}